\documentclass[a4paper,12pt]{amsart}
\usepackage{amsmath,amssymb,amsthm,verbatim}
\usepackage[mathscr]{eucal}
\usepackage{mathpazo}
\usepackage[T1]{fontenc}

\theoremstyle{plain}
\newtheorem{theorem}{Theorem}
\newtheorem{lemma}[theorem]{Lemma}
\newtheorem{proposition}[theorem]{Proposition}

\theoremstyle{definition}
\newtheorem*{remark}{Remark}
\newtheorem*{lemma'}{Lemma 5'}
\newtheorem*{proposition'1}{Proposition 6'}
\newtheorem*{proposition'2}{Proposition 7'}
\newtheorem*{theorem'}{Theorem 8'}

\newcommand{\M}{\mathscr{M}}
\newcommand{\N}{\mathscr{N}}
\newcommand{\D}{\mathcal{D}}
\newcommand{\Ha}{\mathcal{H}}
\newcommand{\1}{\mathbb{1}}
\newcommand{\om}{\omega}
\newcommand{\ro}{\rho}

\newcommand{\f}{\varphi}

\newcommand{\tr}{\operatorname{tr}}
\newcommand{\s}{\operatorname{s}}
\renewcommand{\geq}{\geqslant}
\renewcommand{\leq}{\leqslant}
\DeclareMathOperator{\id}{id}

\begin{document}
\title[Relative modular operator]{Relative modular operator in semifinite von Neumann algebras and its use}
\author{Andrzej \L uczak, Hanna Pods\k{e}dkowska, Rafa{\l} Wieczorek}
\address{Faculty of Mathematics and Computer Science\\
         \L\'od\'z University\\
         ul. S. Banacha 22\\
         90-238 \L\'od\'z, Poland}

\email[Andrzej \L uczak]{andrzej.luczak@wmii.uni.lodz.pl}
\email[Hanna Pods\k{e}dkowska]{hanna.podsedkowska@wmii.uni.lodz.pl}
\email[Rafa{\l} Wieczorek]{rafal.wieczorek@wmii.uni.lodz.pl}
\date{}
\begin{abstract}
 We present some results concerning the relative modular operator in semifinite von Neumann algebras. These results allow one to prove some basic formula for trace, to obtain equivalence between Araki's relative entropy and Umegaki's information as well as to derive some formulae for quasi-entropies, and R\'enyi's relative entropy known in finite dimension.
\end{abstract}
\maketitle

\section*{Introduction}
In the paper, we investigate the relative modular operator in semifinite von Neumann algebras. In finite dimension, this operator is bounded and expressed in an easy way by means of the density operators. In infinite dimension, the relative modular operator is unbounded and its connection with the density operators which can, in general, be also unbounded, remained unclear. In Sections \ref{MO} and \ref{MO1}, this connection is established giving a compact formula for the relative modular operator in terms of the density operators. The main points of the analysis are presented in Section \ref{MO} where in order to avoid cumbersome technicalities a faithfulness assumption is made. This assumption is dropped in Section~ \ref{MO1} where taking advantage of the analysis performed in the previous section, the formula for the relative modular operator is obtained in full generality. This formula is then used in Section \ref{RE-I} to prove some basic formula for trace, to obtain equality between the relative entropy and the information as well as to derive some formulae for quasi-entropies, and R\'enyi's relative entropy generalising thus the results known in finite dimension to an arbitrary semifinite von Neumann algebra.

At this point one thing must be mentioned. In the general case of an arbitrary von Neumann algebra, it is possible and fruitful to represent this algebra on Haagerup's $L^2$-space, and consider the relative modular operator on this space (cf. \cite{Hi}). Then one obtains formulae which are counterparts of the formulae in this paper, in particular, the formula \eqref{basic1} is \emph{mutatis mutandis} the same as the formula (12.9) in \cite{Hi}. Those formulae are obtained by means of sophisticated techniques involving e.g. analytical continuation or the  `$2\times2$-matrix procedure' together with using the balanced weight. Our formulae are obtained in a much simpler way, but the main point for the presentation of their derivation lies in the following: the crucial formulae for the relative modular operator are (12') and \eqref{basic2} --- they allow one to obtain clear and compact forms of a number of relations in which this operator is involved. However, (12') and \eqref{basic2} require an analysis which can not be performed in Haagerup's $L^2$-space and which is one of the core results of the paper.

\section{Preliminaries and notation}
Let $\M$ be a semifinite von Neumann algebra with a normal semifinite faithful trace $\tau$, identity $\1$, and predual $\M_*$. The operator norm on $\M$ shall be denoted by $\|\cdot\|_\infty$. By $\M^+$ we shall denote the set of positive operators in $\M$, and by $\M_*^+$ --- the set of positive functionals in $\M_*$. These functionals will be sometimes referred to as (non-normalised) states.

The algebra of \emph{measurable operators} $\widetilde{\M}$ is defined as a topological ${}^*$-algebra of densely defined closed operators affiliated with $\M$ with strong addition $\dotplus$ and strong multiplication $\cdot$, i.e.
\[
 x\dotplus y=\overline{x+y},\qquad x\cdot y=\overline{xy},\qquad x,y\in\widetilde{\M},
\]
where $\overline{x+y}$ and $\overline{xy}$ are the closures of the corresponding operators defined by addition and composition, respectively, on the natural domains given by the intersections of the domains of $x$ and $y$ and of the range of $y$ and the domain of $x$, respectively (in what follows, while dealing with operators in $\widetilde{\M}$ we shall simply write $x+y$ instead of $x\dotplus y$, and $xy$ instead of $x\cdot y$). The translation-invariant measure topology is defined by a fundamental system of neighbourhoods of $0$, $\{N(\varepsilon,\delta): \varepsilon,\delta>0\}$, given by
\begin{align*}
 N(\varepsilon,\delta)=\{x\in\widetilde{\M}:&\text{ there exists a projection $p$ in $\M$ such that}\\ &xp\in\M,\quad\|xp\|_\infty\leq\varepsilon\quad\text{and} \quad\tau(p^\bot)\leq\delta\}.
\end{align*}
Thus for operators $x_n, x\in\widetilde{\M}$, the sequence $(x_n)$ converges to $x$ \emph{in measure} if for any $\varepsilon,\delta>0$ there exists $n_0$ such that for each $n\geq n_0$ there exists a projection $p_n\in\M$ such that
\[
 \tau(p_n^\bot)\leq\delta,\qquad (x_n-x)p_n\in\M,\qquad
 \text{and}\qquad\|(x_n-x)p_n\|_\infty\leq\varepsilon.
\]
The following `technical' form of convergence in measure proved in \cite[Proposition 2.7]{Y} is useful. Let
\[
 |x_n-x|=\int_0^{\infty}\lambda\,e_n(d\lambda)
\]
be the spectral decomposition of $|x_n-x|$ with spectral measure $e_n$ taking values in $\M$ since $x_n-x$, and thus $|x_n-x|$, are affiliated with $\M$. Then $x_n\to x$ \emph{in measure} if and only if for each $\varepsilon>0$
\[
 \tau(e_n([\varepsilon,\infty)))\to 0.
\]

A sequence $(x_n)$ of operators in $\widetilde{\M}$ is said to converge to $x\in\widetilde{\M}$ \emph{in Segal's sense} if for each $\varepsilon>0$ there is a projection $p\in\M$ such that $\tau(p^\bot)<\varepsilon$, $(x_n-x)p\in\M$ for sufficiently large $n$, and
\[
 \|(x_n-x)p\|_\infty\to0.
\]
It is clear that Segal's convergence implies convergence in measure. (A short intermezzo is probably in order here. The term \emph{Segal's convergence} was introduced by E.C. Lance in \cite{L} in honour of I. Segal, however, Segal himself did not consider this mode of convergence in \cite{S}, restricting attention to so called \emph{convergence nearly everywhere} which is weaker. For $\M$ finite, Segal's convergence and convergence nearly everywhere are equivalent as well as equivalent are convergences \emph{almost uniform} or \emph{closely on large sets} defined in \cite{B}. The notion of Segal's convergence leads in a natural way to the notion of \emph{Segal's continuity} which for `noncommutative stochastic processes' was considered in \cite{GL, GL1,Lu}.)

For each $\rho\in\M_*$, there is a measurable operator $h$ such that
\[
 \rho(x)=\tau(xh)=\tau(hx), \quad x\in\M.
\]
The space of all such operators is denoted by $L^1(\M,\tau)$, and the correspondence above is one-to-one and isometric, where the norm on $L^1(\M,\tau)$, denoted by $\|\cdot\|_1$, is defined as
\[
 \|h\|_1=\tau(|h|), \quad h\in L^1(\M,\tau).
\]
The space of all measurable operators $h$ such that $\tau(|h|^p)<+\infty$, $p\geq1$, constitutes a Banach space $L^p(\M,\tau)$ with the norm
\[
 \|h\|_p=\tau(|h|^p)^{\frac{1}{p}}.
\]
(In the theory of noncommutative $L^p$-spaces for semifinite von Neumann algebras, it it shown that $\tau$ can be extended to the $h$'s as above; see  e.g. \cite{N,S,Ta1,Ta2,Y} for a detailed account of this theory.) Moreover, to hermitian functionals in $\M_*$ correspond selfadjoint operators in $L^1(\M,\tau)$, and to states in $\M_*$ --- positive operators in $L^1(\M,\tau)$. For a state $\rho$, the corresponding element in $L^1(\M,\tau)^+$ will be denoted by $h_\rho$ and called the \emph{density} of $\rho$, thus
\[
 \rho(x)=\tau(xh_\rho)=\tau(h_\rho x)=\tau\big(h_\rho^{\frac{1}{2}}xh_\rho^{\frac{1}{2}}\big), \quad x\in\M.
\]
In particular,
\[
 \tau(h_\rho)=\rho(\1),
\]
so for the normalised states, we have for their densities the equality $\tau(h_\rho)=1$.

Let $r$ be such that $\frac{1}{p}+\frac{1}{q}=\frac{1}{r}$, and let $x\in L^p(\M,\tau)$, $y\in L^q(\M,\tau)$. Then $xy\in L^r(\M,\tau)$ and the following H\"{o}lder inequality holds
\[
 \|xy\|_r\leq\|x\|_p\|y\|_q.
\]

For an arbitrary $x\in L^p(\M,\tau)$, we have the spectral decomposition
\[
 |x|^p=\int_0^{\infty}\lambda^p\,e(d\lambda).
\]
Thus for any $\varepsilon>0$, we get
\[
 |x|^p\geq\int_{\varepsilon}^{\infty}\lambda^p\,e(d\lambda)\geq\int_{\varepsilon}^{\infty}
 \varepsilon^p\,e(d\lambda)=\varepsilon^p\,e([\varepsilon,\infty)).
\]
Consequently, we obtain the Chebyschev inequality
\[
 \tau(e([\varepsilon,\infty)))\leq\frac{\tau(|x|^p)}{\varepsilon^p}=\frac{\|x\|_p^p}{\varepsilon^p}.
\]
Taking into account the above-mentioned `technical' form of convergence in measure, we have
\begin{lemma}
If a sequence $(x_n)$ of operators in $L^p(\M,\tau)$ converges in $\|\cdot\|_p$-norm, then it converges in measure.
\end{lemma}

For a linear operator $T$ on a Hilbert space $\Ha$, $\D(T)$ will stand for its domain and $\Gamma(T)$ for its graph. We start with a simple fact.
\begin{lemma}\label{clop}
Let $T$ be a closable operator, and let $U$ be unitary. Then
\[
 \overline{UT}=U\overline{T}.
\]
\end{lemma}
\begin{proof}
First we show that if $U$ is unitary and $T$ is closed, then $UT$ is closed. To this end, let $\Gamma(UT)\ni(\xi_n,\eta_n)\to(\xi,\eta)$. We have\\ $\eta_n=UT\xi_n$, so
\[
 UT\xi_n=\eta_n\to\eta,
\]
hence
\[
 T\xi_n\to U^*\eta,
\]
which yields
\[
 \Gamma(T)\ni(\xi_n,T\xi_n)\to(\xi,U^*\eta).
\]
From the closedness of $T$, it follows that $\xi\in\Gamma(T)$ and $U^*\eta=T\xi$, hence $\eta=UT\xi$, i.e. $(\xi,\eta)\in\Gamma(UT)$.

Now let $T$ be closable. Since $UT\subset U\overline{T}$, we infer from the first part of the proof that
\[
 \overline{UT}\subset\overline{U\overline{T}}=U\overline{T}.
\]
Consequently, we have
\[
 \overline{T}=\overline{U^*UT}\subset U^*\overline{UT}\subset U^*U\overline{T}=\overline{T},
\]
which means that
\[
 \overline{T}=U^*\overline{UT},
\]
i.e.
\[
 U\overline{T}=\overline{UT}. \qedhere
\]
\end{proof}
Let $\tau$ be a normal semifinite faithful trace on $\M$. Its definition ideal $\mathfrak{M}$ is defined as a linear span of all positive elements in $\M$ of finite trace, so the elements in $\mathfrak{M}$ are exactly those that have finite trace. By definition, the semifiniteness of $\tau$ means that its definition ideal $\mathfrak{M}$ is $\sigma$-dense in $\M$. We are interested in the set $\M\cap L^2(\M,\tau)$.

Recall that the $\sigma$-strong* topology on $\M$ is given by the family of seminorms
\[
 \M\ni x\mapsto(\ro(x^*x)+\ro(xx^*))^{\frac{1}{2}}, \quad \ro\in\M_*^+.
\]

\begin{lemma}\label{sstrong}
Let $\tau$ be a normal semifinite faithful trace on $\M$. Then \linebreak $\M\cap L^2(\M,\tau)$ is $\sigma$-strong* dense in $\M$.
\end{lemma}
\begin{proof}
First we shall show that for an arbitrary projection $e\in\M$ we have
\[
 e=\sum_ie_i,
\]
where $e_i$ are pairwise orthogonal projections in $\M$ such that \linebreak $\tau(e_i)<\infty$. On account of \cite[Proposition 8.5.2]{KR}, there exists a nonzero projection $f\in\M$ and a positive number $\alpha$ such that $0<\tau(f)<\infty$ and $\alpha f\leq e$. It follows that $\alpha\leq1$ and $f\leq e$. Let $\mathfrak{F}=\{e_i\}$ be a maximal family of nonzero pairwise orthogonal projections such that $e_i\leq e$ and $\tau(e_i)<\infty$. From the maximality of $\mathfrak{F}$ and the faithfulness of $\tau$, it follows that $\displaystyle{\sum_ie_i=e}$.

Further, we have
\[
 e=\lim_{F\text{-finite}}e_F,
\]
where
\[
 e_F=\sum_{i\in F}e_i.
\]
$e_F$ are projections such that $\tau(e_F)<\infty$, and $e_F\to e$ \emph{strongly}, and since $e-e_F$ are projections, $e_F\to e$\; $\sigma$-\emph{strongly}*.

Let now $x\in\M^+$. From the spectral theorem, it follows that $x$ may be arbitrarily closely in norm (and hence in the $\sigma$-strong* topology) approximated from below by operators of the form $\displaystyle{\sum_{i=1}^m\lambda_ie_i}$, where $\lambda_i>0$ and $e_i$ are projections in $\M$. Taking projections $e'_i\leq e_i$ such that $\tau(e'_i)<\infty$ and $e'_i$ are arbitrarily close to $e_i$ in the $\sigma$-strong* topology, which is possible by the first part of the proof, we obtain that $\displaystyle{\sum_{i=1}^m\lambda_ie'_i}$ is arbitrarily close to $x$ in the $\sigma$-strong* topology, $\displaystyle{\sum_{i=1}^m\lambda_ie'_i}\leq x$, and
\[
 \tau\Big(\sum_{i=1}^m\lambda_ie'_i\Big)<\infty.
\]
This means that for each $x\in\M^+$ we can find a net $\{x_i\}\subset\mathfrak{M}^+$ such that $x_i\leq x$ and $x_i\to x$ $\sigma$-\emph{strongly}*.

Let again $x$ be an arbitrary element in $\M^+$. From what we have proved, there is a net $\{x_i\}\subset\mathfrak{M}^+$ such that $x_i\leq x^2$ and $x_i\to x^2$ $\sigma$-\emph{strongly}*. It follows that $x_i^{\frac{1}{2}}\in\M\cap L^2(\M,\tau)$, $x_i^{\frac{1}{2}}\leq x$, and the boundedness of $\big\{x_i^{\frac{1}{2}}\big\}$ yields  $x_i^{\frac{1}{2}}\to x$ \emph{strongly}. The boundedness and positivity of $x_i^{\frac{1}{2}}$ and $x$ give $x_i^{\frac{1}{2}}\to x$ $\sigma$-\emph{strongly}*. Thus we have proved that $\M^+\cap L^2(\M,\tau)$ is $\sigma$-strong* dense in $\M^+$, and the decomposition of an arbitrary $x\in\M$ as a linear combination of four positive elements yields the claim.
\end{proof}

\section{The fundamental representations of the algebra}\label{rep}
The following construction will be crucial in our further considerations. The space $L^2(\M,\tau)$ consists of (possibly unbounded) linear operators affiliated with $\M$ such that for
$a\in L^2(\M,\tau)$ we have $\tau(a^*a)=\tau(|a|^2)<\infty$. With a scalar product defined on $L^2(\M,\tau)$ by the formula
\begin{equation}\label{sp}
 \langle a|b\rangle_2=\tau(a^*b), \quad a,b\in L^2(\M,\tau),
\end{equation}
$L^2(\M,\tau)$ becomes a Hilbert space which we shall denote by $\Ha$. The operators $a\in L^2(\M,\tau)$, treated as elements of $\Ha$, shall be denoted by $\Lambda(a)$, thus formula \eqref{sp} reads
\[
 \langle a|b\rangle_2=\langle\Lambda(a)|\Lambda(b)\rangle_{\Ha}=\tau(a^*b), \quad a,b\in L^2(\M,\tau),
\]
(note that $a^*b$ denotes the product of operators in $L^2(\M,\tau)$). On the space $L^2(\M,\tau)$ we shall also consider a norm $\|\cdot\|_2$  defined as
\[
 \|a\|_2=\big(\tau(|a|^2)\big)^{\frac{1}{2}},
\]
so we have
\[
 \|a\|_2=\|\Lambda(a)\|_{\Ha}.
\]

We define a representation $\pi$ of $\M$ on $\Ha$ by the formula
\[
 \pi(x)\Lambda(a)=\Lambda(xa), \quad x\in\M,\,a\in L^2(\M,\tau),
\]
and an antirepresentation $\pi'$ of $\M$ on $\Ha$ by the formula
\[
 \pi'(x)\Lambda(a)=\Lambda(ax), \quad x\in\M,\,a\in L^2(\M,\tau).
\]
It is known that $\pi$ and $\pi'$ are normal faithful, and that $\pi(\M)'=\pi'(\M)$ (cf. \cite[Theorem V.2.22]{Ta1}). Let $x$ be a selfadjoint operator affiliated with $\M$ with spectral decomposition
\begin{equation}\label{sp1}
 x=\int_{-\infty}^\infty \lambda\,e(d\lambda).
\end{equation}
We define $\pi(x)$ and $\pi'(x)$ by the formulae
\[
 \pi(x)=\int_{-\infty}^\infty \lambda\,\pi(e(d\lambda)), \qquad \pi'(x)=\int_{-\infty}^\infty \lambda\,\pi'(e(d\lambda)),
\]
(since $\pi$ and $\pi'$ are normal, $\pi(e(\cdot))$ and $\pi'(e(\cdot))$ are spectral measures). It is clear that $\pi(x)$ and $\pi'(x)$ are selfadjoint operators affiliated with $\pi(\M)$ and $\pi'(\M)$, respectively. Let $f\colon\mathbb{R}\to\mathbb{R}$ be a Borel function. For selfadjoint $x$ with spectral decomposition \eqref{sp1}, we have, using `integration by image measure',
\begin{align*}
 \pi(f(x))&=\pi\bigg(\int_{-\infty}^\infty f(\lambda)\,e(d\lambda)\bigg)\\
 &=\pi\bigg(\int_{-\infty}^\infty t\,(f\circ e)(dt)\bigg)=\int_{-\infty}^\infty t\,\pi((f\circ e)(dt)),
\end{align*}
where $f\circ e$ is a spectral measure defined as
\[
 (f\circ e)(E)=e(f^{-1}(E)).
\]
On the other hand, by the same token
\[
 f(\pi(x))=\int_{-\infty}^\infty f(\lambda)\,\pi(e(d\lambda))=\int_{-\infty}^\infty t\,(f\circ(\pi(e(dt)))),
\]
where for the spectral measure $f\circ(\pi(e(\cdot)))$ we have
\[
 f\circ(\pi(e(E)))=\pi(e(f^{-1}(E))).
\]
It follows that $\pi(f(x))$ and $f(\pi(x))$ have the same spectral measures, yielding the equality
\begin{equation}\label{pi1}
 \pi(f(x))=f(\pi(x)).
\end{equation}
In the same way we obtain the equality
\begin{equation}\label{pi'1}
 \pi'(f(x))=f(\pi'(x)).
\end{equation}
Now we want to describe the action of $\pi(x)$ and $\pi'(x)$ for measurable $x$ as selfadjoint, possibly unbounded, operators on $\Ha$.
\begin{proposition}\label{P}
Let $x^*=x\in\widetilde{\M}$, and let $\Lambda(a)\in\D(\pi(x))$. Then\\ $xa\in L^2(\M,\tau)$, and
\[
 \pi(x)\Lambda(a)=\Lambda(xa).
\]
\end{proposition}
\begin{proof}
For the nonnegative measure
\[
 \|\pi(e(\cdot))\Lambda(a)\|_{\Ha}^2=\langle\Lambda(a)|\pi(e(\cdot))\Lambda(a)\rangle_{\Ha},
\]
we have
\begin{align*}
 \|\pi(e(E))\Lambda(a)\|_{\Ha}^2&=\|\Lambda(e(E)a)\|_{\Ha}^2\\
 &=\tau(a^*e(E)a)=\tau(e(E)aa^*)=\ro(E),
\end{align*}
where $\ro=(aa^*)\tau\in\M_*^+$ (i.e. $aa^*$ is the density of $\ro$), consequently,
\begin{equation}\label{fin}
 \int_{-\infty}^\infty \lambda^2\,\ro(e(d\lambda))=\int_{-\infty}^\infty \lambda^2\,\|\pi(e(d\lambda))\Lambda(a)\|_{\Ha}^2<\infty.
\end{equation}
For $x$ with spectral decomposition \eqref{sp1}, define its truncation $x_{[n]}$ by
\begin{equation}\label{trunc}
 x_{[n]}=\int_{-n}^n \lambda\,e(d\lambda).
\end{equation}
Then for $\xi\in\D(x)$, we have
\begin{equation}\label{trunc1}
 \|x\xi-x_{[n]}\xi\|^2=\int_{-\infty}^{-n}\lambda^2\,\|e(d\lambda)\xi\|^2+\int_n^\infty\lambda^2\,\|e(d\lambda)\xi\|^2\underset{n\to\infty}{\longrightarrow}0,
\end{equation}
since
\[
 \int_{-\infty}^\infty \lambda^2\,\|e(d\lambda)\xi\|^2<\infty.
\]

For $n>m$, we have by virtue of \eqref{fin}
\begin{align*}
 \|\Lambda(x_{[n]}a)-\Lambda(x_{[m]}a)\|_{\Ha}^2&=\tau(a^*|x_{[n]}-x_{[m]}|^2a)=\ro(|x_{[n]}-x_{[m]}|^2)\\
 &=\ro\bigg(\int_{-n}^{-m}\lambda^2\,e(d\lambda)+\int_m^n\lambda^2\,e(d\lambda)\bigg)\\
 &=\int_{-n}^{-m}\lambda^2\,\ro(e(d\lambda))+\int_m^n\lambda^2\,\ro(e(d\lambda))\underset{m,n\to\infty}{\longrightarrow}0,
\end{align*}
thus
\[
 \|x_{[n]}a-x_{[m]}a\|_2^2=\|\Lambda(x_{[n]}a)-\Lambda(x_{[m]}a)\|_{\Ha}\underset{m,n\to\infty}{\longrightarrow}0,
\]
which means that the sequence $(x_{[n]}a)$ in $L^2(\M,\tau)$ is Cauchy in $\|\cdot\|_2$-norm. So there is a $z\in L^2(\M,\tau)$ such that
\[
 \|x_{[n]}a-z\|_2\to0.
\]
In particular, $x_{[n]}a\to z$ \emph{in measure}. We have the formula
\[
 x-x_{[n]}=\int_{-\infty}^{-n}\lambda\,e(d\lambda)+\int_n^\infty\lambda\,e(d\lambda).
\]
For arbitrary $\varepsilon>0$, we can find $\alpha>0$ such that
\[
 \tau(e([-\alpha,\alpha])^\bot)=\tau(e((-\infty,-\alpha)\cup(\alpha,\infty)))<\varepsilon,
\]
and taking $n>\alpha$, we obtain
\[
 \|(x-x_{[n]})e([-\alpha,\alpha])\|_\infty=0,
\]
which means that $x_{[n]}\to x$ \emph{in Segal's sense}. In particular, $x_{[n]}\to x$ \emph{in measure}, and thus $x_{[n]}a\to xa$ \emph{in measure}, giving $z=xa\in L^2(\M,\tau)$, and
\begin{equation}\label{lam}
 \|\Lambda(x_{[n]}a)-\Lambda(xa)\|_{\Ha}=\|x_{[n]}a-xa\|_2\to0.
\end{equation}
Now we obtain, on account of the relation \eqref{trunc1} and an easily verifiable fact that $\pi(x_{[n]})=\pi(x)_{[n]}$,
\[
 \Lambda(x_{[n]}a)=\pi(x_{[n]})\Lambda(a)=\pi(x)_{[n]}\Lambda(a)\to\pi(x)\Lambda(a),
\]
which together with the formula \eqref{lam} shows the claim.
\end{proof}
In the same way, we get for $\Lambda(a)\in\D(\pi'(x))$ the formula
\[
 \pi'(x)\Lambda(a)=\Lambda(ax).
\]

\section{The relative modular operator (faithful state)}\label{MO}
In this section, we want to find the form of the relative modular operator $\Delta(\f,\om)$ in the space $\Ha$ in terms of the densities $h_\f$ and $h_\om$ of the states $\f$ and $\om$, respectively. For the sake of better readability, in order to avoid some technical complications we assume that $\om$ is faithful. $h_\f$ and $h_\om$ are selfadjoint positive operators in $L^1(\M,\tau)$ such that for $x\in\M$ the following formulae hold
\[
 \f(x)=\tau(h_\f x)=\tau\big(h_\f^{\frac{1}{2}}xh_\f^{\frac{1}{2}}\big), \qquad \om(x)=\tau(h_\om x)=\tau\big(h_\om^{\frac{1}{2}}xh_\om^{\frac{1}{2}}\big),
\]
so in the representation $\pi$ of $\M$ on $\Ha=L^2(\M,\tau)$ we have for $x\in\M$
\begin{equation}\label{f,om}
 \begin{aligned}
  \langle\Lambda\big(h_\f^{\frac{1}{2}}\big)|\pi(x)\Lambda\big(h_\f^{\frac{1}{2}}\big)\rangle_\Ha&=\langle\Lambda\big(h_\f^{\frac{1}{2}}\big)|\Lambda\big(xh_\f^{\frac{1}{2}}\big)\rangle_\Ha\\
  &=\tau\big(h_\f^{\frac{1}{2}}xh_\f^{\frac{1}{2}}\big)=\f(x),\\
  \langle\Lambda\big(h_\om^{\frac{1}{2}}\big)|\pi(x)\Lambda\big(h_\om^{\frac{1}{2}}\big)\rangle_\Ha&=\langle\Lambda\big(h_\om^{\frac{1}{2}}\big)|\Lambda\big(xh_\om^{\frac{1}{2}}\big)\rangle_\Ha\\
  &=\tau\big(h_\om^{\frac{1}{2}}xh_\om^{\frac{1}{2}}\big)=\om(x)
 \end{aligned}
\end{equation}
which means that in this representation $\f$ and $\om$ are vector states with representing vectors $\Lambda\big(h_\f^{\frac{1}{2}}\big)$ and $\Lambda\big(h_\om^{\frac{1}{2}}\big)$, respectively. Similarly, for the antirepresentation $\pi'$ we have
\begin{equation}\label{f,om1}
 \begin{aligned}
  \langle\Lambda\big(h_\f^{\frac{1}{2}}\big)|\pi'(x)\Lambda\big(h_\f^{\frac{1}{2}}\big)\rangle_\Ha&=\langle\Lambda\big(h_\f^{\frac{1}{2}}\big)|\Lambda\big(h_\f^{\frac{1}{2}}x\big)\rangle_\Ha\\
  &=\tau\big(h_\f^{\frac{1}{2}}h_\f^{\frac{1}{2}}x\big)=\f(x),\\
  \langle\Lambda\big(h_\om^{\frac{1}{2}}\big)|\pi'(x)\Lambda\big(h_\om^{\frac{1}{2}}\big)\rangle_\Ha&=\langle\Lambda\big(h_\om^{\frac{1}{2}}\big)|\Lambda\big(h_\om^{\frac{1}{2}}x\big)\rangle_\Ha\\
  &=\tau\big(h_\om^{\frac{1}{2}}h_\om^{\frac{1}{2}}x\big)=\om(x),
 \end{aligned}
\end{equation}
which means that also in the antirepresentation $\pi'$, $\f$ and $\om$ are vector states with representing vectors $\Lambda\big(h_\f^{\frac{1}{2}}\big)$ and $\Lambda\big(h_\om^{\frac{1}{2}}\big)$, respectively.

Following Araki \cite{A}, we define an antilinear operator $S$ on the space
\[
 \D(S)=\{\pi(x)\Lambda\big(h_\om^{\frac{1}{2}}\big):x\in\M\}=\{\Lambda\big(xh_\om^{\frac{1}{2}}\big):x\in\M\}
\]
by the formula
\[
 S\big(\Lambda\big(xh_\om^{\frac{1}{2}}\big)\big)=\pi(x)^*\Lambda\big(h_\f^{\frac{1}{2}}\big)=\Lambda\big(x^*h_\f^{\frac{1}{2}}\big).
\]
Since $\om$ is faithful, it follows that $S$ is densely defined, moreover, it is closable. The relative modular operator is then defined as
\[
 \Delta(\f,\om)=S^*\overline{S}.
\]
It is selfadjoint and positive, and the following polar decomposition holds
\[
 \overline{S}=J\Delta(\f,\om)^{\frac{1}{2}},
\]
where $J$ is an antilinear isometry such that $J^2=\id_{\Ha}$. Put
\[
 \D(S_0)=\{\Lambda\big(xh_\om^{\frac{1}{2}}\big):x\in \M\cap L^2(\M,\tau)\}, \qquad S_0=S|\D(S_0).
\]
\begin{lemma}\label{S_0}
We have
\[
 \overline{S_0}=\overline{S}.
\]
\end{lemma}
\begin{proof}
Take an arbitrary $(\xi,\eta)\in\Gamma(S)$, $\xi=\Lambda\big(xh_\om^{\frac{1}{2}}\big)$, $\eta=\Lambda\big(x^*h_\f^{\frac{1}{2}}\big)$, and let $\{x_i\}$ be a net in $\M\cap L^2(\M,\tau)$ $\sigma$-strongly* convergent to $x$. We have
\begin{align*}
 \|\Lambda\big(x_ih_\om^{\frac{1}{2}}\big)-\Lambda\big(xh_\om^{\frac{1}{2}}\big)\|_{\Ha}^2&=\tau\big(h_\om^{\frac{1}{2}}|x_i-x|^2h_\om^{\frac{1}{2}}\big)\\
 &=\om(|x_i-x|^2)\to0,
\end{align*}
and
\begin{align*}
 \|\Lambda\big(x_i^*h_\f^{\frac{1}{2}}\big)-\Lambda\big(x^*h_\f^{\frac{1}{2}}\big)\|_{\Ha}^2&=\tau\big(h_\f^{\frac{1}{2}}|x_i^*-x^*|^2h_\f^{\frac{1}{2}}\big)\\
 &=\f(|x_i^*-x^*|^2)\to0,
\end{align*}
which shows that
\[
 \Gamma(S_0)\ni\big(\Lambda\big(x_ih_\om^{\frac{1}{2}}\big),\Lambda\big(x_i^*h_\f^{\frac{1}{2}}\big)\big)\to\big(\Lambda\big(xh_\om^{\frac{1}{2}}\big),\Lambda\big(x^*h_\f^{\frac{1}{2}}\big)\big)=(\xi,\eta).
\]
This means that $\Gamma(S)\subset\overline{\Gamma(S_0)}$, consequently $S\subset\overline{S_0}$, i.e. $\overline{S}\subset\overline{S_0}$. Since obviously $\overline{S_0}\subset\overline{S}$, the conclusion follows.
\end{proof}
From the definition of $\Delta(\f,\om)$, it follows that
\[
 \overline{\Delta(\f,\om)^{\frac{1}{2}}|\D(S)}=\Delta(\f,\om)^{\frac{1}{2}},
\]
however, we get more.
\begin{proposition}\label{D}
The following formula holds
\[
 \overline{\Delta(\f,\om)^{\frac{1}{2}}|\D(S_0)}=\Delta(\f,\om)^{\frac{1}{2}}.
\]
\end{proposition}
\begin{proof}
We have
\[
 J\Delta(\f,\om)^{\frac{1}{2}}|\D(S_0)=\overline{S}|\D(S_0)=S_0,
\]
hence
\[
 \overline{J\Delta(\f,\om)^{\frac{1}{2}}|\D(S_0)}=\overline{S_0}=\overline{S}=J\Delta(\f,\om)^{\frac{1}{2}}.
\]
Since $J$ is antiunitary and $J^2=\id_{\Ha}$, we obtain, applying $J$ to both sides of the above equality and taking into account Lemma \ref{clop},
\begin{align*}
 \Delta(\f,\om)^{\frac{1}{2}}&=J\big(\overline{J\Delta(\f,\om)^{\frac{1}{2}}|\D(S_0)}\big)\\
 &=\overline{J^2\Delta(\f,\om)^{\frac{1}{2}}|\D(S_0)}=\overline{\Delta(\f,\om)^{\frac{1}{2}}|\D(S_0)}. \qedhere
\end{align*}
\end{proof}
Our next aim is to relate $\Delta(\f,\om)^{\frac{1}{2}}$ to the operator\\ $\pi\big(h_\f^{\frac{1}{2}}\big)\pi'\big(h_\om^{-\frac{1}{2}}\big)$. We start with
\begin{proposition}\label{Delta0}
The following formula holds
\[
 \pi(h_\f^{\frac{1}{2}})\pi'(h_\om^{-\frac{1}{2}})|\D(S_0)=\Delta(\f,\om)^{\frac{1}{2}}|\D(S_0).
\]
\end{proposition}
\begin{proof}
Observe first that $\pi'\big(h_\om^{-\frac{1}{2}}\big)=\pi'\big(h_\om^{\frac{1}{2}}\big)^{-1}$, so
\[
 \pi'(h_\om^{-\frac{1}{2}})\pi'(h_\om^{\frac{1}{2}})=\id|\D\big(\pi'\big(h_\om^{\frac{1}{2}}\big)\big).
\]
Since for $x\in\M\cap L^2(\M,\tau)$, we have $\Lambda(x)\in\D\big(\pi'\big(h_\om^{\frac{1}{2}}\big)\big)$ and
\[
 \pi'\big(h_\om^{\frac{1}{2}}\big)\Lambda(x)=\Lambda\big(xh_\om^{\frac{1}{2}}\big),
\]
it follows that $\Lambda\big(xh_\om^{\frac{1}{2}}\big)\in\D(\pi'\big(h_\om^{-\frac{1}{2}}\big)$ and
\[
 \pi'\big(h_\om^{-\frac{1}{2}}\big)\Lambda\big(xh_\om^{\frac{1}{2}}\big)=\Lambda(x).
\]
For $x$ as above, we have $xh_\f^{\frac{1}{2}}\in L^2(\M,\tau)$, so
\[
 \pi(h_\f^{\frac{1}{2}})\pi'(h_\om^{-\frac{1}{2}})\Lambda(xh_\om^{\frac{1}{2}})=\pi(h_\f^{\frac{1}{2}})\Lambda(x)=\Lambda(h_\f^{\frac{1}{2}}x).
\]
Let $h\geq0$ be an arbitrary element in $L^1(\M,\tau)$. Then $h^{\frac{1}{2}}\in L^2(\M,\tau)$, moreover, all positive elements in $L^2(\M,\tau)$ arise in this way. The set $\{\Lambda(h^{\frac{1}{2}}):h\in L^1(\M,\tau), h\geq0\}$ is a positive cone in $\Ha$, and we have
\begin{align*}
 \langle\Lambda(h^{\frac{1}{2}})|\pi(h_\f^{\frac{1}{2}})\pi'(h_\om^{-\frac{1}{2}})\Lambda(xh_\om^{\frac{1}{2}})\rangle_{\Ha}&=\langle\Lambda(h^{\frac{1}{2}})|\Lambda(h_\f^{\frac{1}{2}}x)\rangle_{\Ha}\\
 &=\tau\big(h^{\frac{1}{2}}h_\f^{\frac{1}{2}}x\big).
\end{align*}
On the other hand, from modular theory it follows that the isometry $J$ on this cone is identity, consequently, we get
\begin{align*}
 &\langle\Lambda(h^{\frac{1}{2}})|\Delta(\f,\om)^{\frac{1}{2}}\Lambda(xh_\om^{\frac{1}{2}})\rangle_{\Ha}
 =\overline{\langle J(\Lambda(h^{\frac{1}{2}}))|J\Delta(\f,\om)^{\frac{1}{2}}\Lambda(xh_\om^{\frac{1}{2}})\rangle}_{\Ha}\\
 =&\overline{\langle\Lambda(h^{\frac{1}{2}})|S\big(\Lambda\big(xh_\om^{\frac{1}{2}}\big)\big)\rangle}_{\Ha}=\langle\Lambda\big(x^*h_\f^{\frac{1}{2}}\big)|\Lambda(h^{\frac{1}{2}})\rangle_{\Ha}\\
 =&\tau\big(h_\f^{\frac{1}{2}}xh^{\frac{1}{2}}\big)=\tau\big(h^{\frac{1}{2}}h_\f^{\frac{1}{2}}x\big),
\end{align*}
where the change of order under the trace in the last equality is justified by the fact that $h^{\frac{1}{2}},h_\f^{\frac{1}{2}}x\in L^2(\M,\tau)$. It follows that
\[
 \langle\Lambda(h^{\frac{1}{2}})|\pi(h_\f^{\frac{1}{2}})\pi'(h_\om^{-\frac{1}{2}})\Lambda(xh_\om^{\frac{1}{2}})\rangle_\Ha
 =\langle\Lambda(h^{\frac{1}{2}})|\Delta(\f,\om)^{\frac{1}{2}}\Lambda(xh_\om^{\frac{1}{2}})\rangle_{\Ha},
\]
and since the positive cone spans the whole of $\Ha$, we get
\[
 \pi(h_\f^{\frac{1}{2}})\pi'(h_\om^{-\frac{1}{2}})\Lambda(xh_\om^{\frac{1}{2}})=\Delta(\f,\om)^{\frac{1}{2}}\Lambda(xh_\om^{\frac{1}{2}}),
\]
proving the claim.
\end{proof}
As a corollary we obtain, taking into account Proposition \ref{D}, the formula
\[
 \overline{\pi(h_\f^{\frac{1}{2}})\pi'(h_\om^{-\frac{1}{2}})|\D(S_0)}=\Delta(\f,\om)^{\frac{1}{2}}.
\]
The selfadjoint positive operators $\pi(h_\f^{\frac{1}{2}})$ and $\pi'(h_\om^{-\frac{1}{2}})$ commute, that is their spectral measures commute, thus there is a spectral measure $m$ in $\Ha$, and nonnegative Borel functions $u$ and $v$ (in fact, $v$ is even positive) such that
\begin{equation}\label{pi}
 \pi(h_\f^{\frac{1}{2}})=\int_{-\infty}^\infty u(\lambda)\,m(d\lambda), \quad \pi'(h_\om^{-\frac{1}{2}})=\int_{-\infty}^\infty v(\lambda)\,m(d\lambda).
\end{equation}
Denote for brevity
\[
 A=\int_{-\infty}^\infty u(\lambda)v(\lambda)\,m(d\lambda).
\]
Then $A$ is a selfadjoint positive operator such that
\[
 \pi(h_\f^{\frac{1}{2}})\pi'(h_\om^{-\frac{1}{2}})\subset A.
\]
\begin{theorem}\label{Delta1}
Let $\pi(h_\f^{\frac{1}{2}})$ and $\pi'(h_\om^{-\frac{1}{2}})$ be represented by the formula \eqref{pi}. Then
\begin{equation}\label{Delta}
 \Delta(\f,\om)^{\frac{1}{2}}=\overline{\pi(h_\f^{\frac{1}{2}})\pi'(h_\om^{-\frac{1}{2}})}=\int_{-\infty}^\infty u(\lambda)v(\lambda)\,m(d\lambda).
\end{equation}
\end{theorem}
\begin{proof}
We have
\[
 \Delta(\f,\om)^{\frac{1}{2}}=\overline{\pi(h_\f^{\frac{1}{2}})\pi'(h_\om^{-\frac{1}{2}})|\D(S_0)}\subset\overline{\pi(h_\f^{\frac{1}{2}})\pi'(h_\om^{-\frac{1}{2}})}\subset A,
\]
and taking adjoints, we get
\[
 A=A^*\subset\Big(\overline{\pi(h_\f^{\frac{1}{2}})\pi'(h_\om^{-\frac{1}{2}})}\Big)^*\subset\Big(\Delta(\f,\om)^{\frac{1}{2}}\Big)^*=\Delta(\f,\om)^{\frac{1}{2}},
\]
i.e.
\[
 \Delta(\f,\om)^{\frac{1}{2}}=A=\overline{\pi(h_\f^{\frac{1}{2}})\pi'(h_\om^{-\frac{1}{2}})}. \qedhere
\]
\end{proof}

\section{The relative modular operator (arbitrary state)}\label{MO1}
In this section, we drop the assumption about the faithfulness of $\om$. To define the relative modular operator, we need some additional notions. First observe that if $h_\om$ is not invertible, then the space $\{\Lambda(xh_\om^{\frac{1}{2}}):x\in\M\}$ is not dense in $\Ha$. Because of the equality $\Lambda(xh_\om^{\frac{1}{2}})=\pi(x)\Lambda(h_\om^{\frac{1}{2}})$, this space is $\pi(\M)$-invariant, so the projection onto it belongs to $\pi(\M)'$, moreover, this projection is the so-called \emph{support} in $\pi(\M)'$ of the vector $\Lambda(h_\om^{\frac{1}{2}})$. On the other hand, if we look at the vector state $\langle\Lambda(h_\om^{\frac{1}{2}})|\cdot\Lambda(h_\om^{\frac{1}{2}})\rangle_\Ha$ restricted to the algebra $\pi(\M)'$, then the support projection of the vector $\Lambda(h_\om^{\frac{1}{2}})$ in $\pi(\M)'$ is nothing else but the support in $\pi(\M)'$ of this vector state.

For a von Neumann algebra $\N$ acting on a Hilbert space $\mathcal{K}$, and a vector state $\ro$ on $\N$ given by a vector $\xi\in\mathcal{K}$, denote by $\s^\N(\ro)=\s^\N(\xi)$ the support of $\ro$ (respectively $\xi$) in $\N$. Since $\pi$ and $\pi'$ are faithful representations, and since $\om$ is in these representations a vector state, we have
\begin{equation}\label{s}
 \begin{aligned}
  \s^{\pi(\M)}(\Lambda(h_\om^{\frac{1}{2}}))&=\pi(\s^\M(\om)),\\
  \s^{\pi(\M)'}(\Lambda(h_\om^{\frac{1}{2}}))&=\s^{\pi'(\M)}(\Lambda(h_\om^{\frac{1}{2}}))=\pi'(\s^\M(\om)),
 \end{aligned}
\end{equation}
and analogously
\begin{equation}\label{s1}
 \begin{aligned}
  \s^{\pi(\M)}(\Lambda(h_\f^{\frac{1}{2}}))&=\pi(\s^\M(\f)),\\
  \s^{\pi(\M)'}(\Lambda(h_\f^{\frac{1}{2}}))&=\s^{\pi'(\M)}(\Lambda(h_\f^{\frac{1}{2}}))=\pi'(\s^\M(\f)).
 \end{aligned}
\end{equation}
Following Araki \cite{A1}, we define the operator $S$ by the formulae
\begin{align*}
 &\D(S)=\{\pi(x)\Lambda\big(h_\om^{\frac{1}{2}}\big):x\in\M\}+\big(\1-\s^{\pi(\M)'}\big(\Lambda\big(h_\om^{\frac{1}{2}}\big)\big)\big)\Ha\\
 =&\{\Lambda(xh_\om^{\frac{1}{2}}):x\in\M\}+\big(\1-\s^{\pi(\M)'}\big(\Lambda\big(h_\om^{\frac{1}{2}}\big)\big)\big)\Ha,\\
 &S(\Lambda(xh_\om^{\frac{1}{2}})+\xi)=\s^{\pi(\M)}\big(\Lambda\big(h_\om^{\frac{1}{2}}\big)\big)\pi(x)^*\Lambda(h_\f^{\frac{1}{2}}),\\ &\xi\in\big(\1-\s^{\pi(\M)'}\big(\Lambda\big(h_\om^{\frac{1}{2}}\big)\big)\big)\Ha.
\end{align*}
It follows that $S$ is a densely defined closable antilinear operator on $\Ha$. Taking into account the relation \eqref{s}, we get
\begin{align*}
 S(\Lambda\big(xh_\om^{\frac{1}{2}}\big)+\xi)&=\s^{\pi(\M)}\pi(x)^*\Lambda\big(h_\f^{\frac{1}{2}}\big)=\pi(\s^\M(\om))\pi(x)^*\Lambda\big(h_\f^{\frac{1}{2}}\big)\\
 &=\pi(\s^\M(\om)x^*)\Lambda\big(h_\f^{\frac{1}{2}}\big)=\Lambda\big(\s^\M(\om)x^*h_\f^{\frac{1}{2}}\big).
\end{align*}
The relative modular operator $\Delta(\f,\om)$ is again defined as
\[
 \Delta(\f,\om)=S^*\overline{S},
\]
and as in Section \ref{MO}, we have
\[
 \overline{S}=J\Delta(\f,\om)^{\frac{1}{2}}.
\]
Now we basically follow the lines of Section \ref{MO} with some necessary refinements. The first step is to define the operator $S_0$:
\begin{align*}
 \D(S_0)&=\{\Lambda\big(xh_\om^{\frac{1}{2}}\big):x\in \M\cap L^2(\M,\tau)\}+\big(\1-\s^{\pi(\M)'}\big(\Lambda\big(h_\om^{\frac{1}{2}}\big)\big)\big)\Ha,\\
 S_0&=S|\D(S_0).
\end{align*}
Now Lemma \ref{S_0} remains unchanged, namely
\begin{lemma'}
We have
\[
 \overline{S_0}=\overline{S}.
\]
\end{lemma'}
\begin{proof}
The only new ingredient in the proof, when compared with the proof of Lemma \ref{S_0}, is taking into account the support of $\om$ which now occurs in the definition of $S_0$. Namely, we have
\begin{align*}
 &\|\Lambda\big(\s^\M(\om)x_i^*h_\f^{\frac{1}{2}}\big)-\Lambda\big(\s^\M(\om)x^*h_\f^{\frac{1}{2}}\big)\|_{\Ha}^2\\
 &=\tau\big(h_\f^{\frac{1}{2}}(x_i-x)\s^\M(\om)(x_i^*-x^*)h_\f^{\frac{1}{2}}\big)\\
 &=\f((x_i-x)\s^\M(\om)(x_i^*-x^*))\leq\f((x_i-x)(x_i^*-x^*))\to0,
\end{align*}
and the rest of the proof is the same as that of Lemma \ref{S_0}.
\end{proof}
As for Proposition \ref{D}, its proof can be repeated word for word, thus we have
\begin{proposition'1}
The following formula holds
\[
 \overline{\Delta(\f,\om)^{\frac{1}{2}}|\D(S_0)}=\Delta(\f,\om)^{\frac{1}{2}}.
\]
\end{proposition'1}
A little more effort is needed for proving a counterpart of Proposition \ref{Delta0}. To this end, we must first find a counterpart of the operator $h_\om^{-\frac{1}{2}}$. Let a function $w\colon\mathbb{R}_+\to\mathbb{R}$ be defined by the formula
\[
 w(\lambda)=\begin{cases}
 \frac{1}{\lambda}, & \text{for $\lambda\ne0$}\\
 0, & \text{for $\lambda=0$}
 \end{cases},
\]
and put
\begin{equation}\label{h_om}
 \widetilde{h_\om}=w(h_\om)=\int_0^\infty w(\lambda)\,e(d\lambda).
\end{equation}
If
\[
 h_\om=\int_0^\infty\lambda\,e(d\lambda)
\]
is the spectral decomposition of $h_\om$, then for the antirepresentation $\pi'$ we have
\begin{align*}
 \pi'(h_\om)&=\int_0^\infty\lambda\,\pi'(e(d\lambda)),\\
 \pi'(\widetilde{h_\om})&=\pi'(w(h_\om))=w(\pi'(h_\om))=\int_0^\infty w(\lambda)\,\pi'(e(d\lambda)),
\end{align*}
hence the operator calculus yields
\begin{align*}
 \pi'\Big(\widetilde{h_\om}^{\frac{1}{2}}\Big)\pi'\big(h_\om^{\frac{1}{2}}\big)&\subset\int_0^\infty w(\lambda)^{\frac{1}{2}}\lambda^{\frac{1}{2}}\,\pi'(e(d\lambda))\\
 &=\pi'(e((0,\infty)))=\s^{\pi'(\M)}(\pi'(h_\om)).
\end{align*}
Put
\[
 \Ha_0=\s^{\pi'(\M)}(\pi'(h_\om))\Ha.
\]
Then $\Ha_0$ is the closure of the range of $\pi'(h_\om)$ as well as the closure of the range of $\pi'(\widetilde{h_\om})$, thus the operators $\pi'\big(h_\om^{\frac{1}{2}}\big)$ and $\pi'\Big(\widetilde{h_\om}^{\frac{1}{2}}\Big)$ are injective on $\Ha_0$; moreover, they send $\Ha_0$ (dense subspaces of) to $\Ha_0$, so we can consider the restricted operators $\pi'\big(h_\om^{\frac{1}{2}}\big)|\Ha_0$ and $\pi'\Big(\widetilde{h_\om}^{\frac{1}{2}}\Big)\Big|\Ha_0$. These operators are selfadjoint, positive, and we have
\[
 \Big(\pi'\Big(\widetilde{h_\om}^{\frac{1}{2}}\Big)\Big|\Ha_0\Big)\big(\pi'\big(h_\om^{\frac{1}{2}}\big)|\Ha_0\big)=\id_{\Ha_0}|\D\big(\pi'\big(h_\om^{\frac{1}{2}}\big)|\Ha_0\big),
\]
which means that
\begin{equation}\label{pi-1}
 \pi'\Big(\widetilde{h_\om}^{\frac{1}{2}}\Big)\Big|\Ha_0=\big(\pi'\big(h_\om^{\frac{1}{2}}\big)|\Ha_0\big)^{-1}.
\end{equation}
Moreover,
\begin{equation}\label{pi0}
 \pi'\Big(\widetilde{h_\om}^{\frac{1}{2}}\Big)\Big|\Ha_0^\bot=\pi'\big(h_\om^{\frac{1}{2}}\big)|\Ha_0^\bot=0.
\end{equation}
Assume that $x\in\M\cap L^2(\M,\tau)$.

Let $E$ be a Borel subset of $(0,\infty)$. Then, since $e(E)\leq e((0,\infty))=\s(\om)$, we have
\begin{align*}
 \pi'(e(E))\Lambda(x\s^\M(\om))&=\Lambda(x\s^\M(\om)e(E))\\
 &=\Lambda(xe(E))=\pi'(e(E))\Lambda(x).
\end{align*}
This yields
\[
 \int_0^\infty\lambda\,\|\pi'(e(d\lambda))\Lambda(x\s^\M(\om))\|^2=\int_0^\infty\lambda\,\|\pi'(e(d\lambda))\Lambda(x)\|^2<\infty,
\]
because $\Lambda(x)\in\D\big(\pi'\big(h_\om^{\frac{1}{2}}\big)\big)$. From this inequality, it follows that $\Lambda(x\s^\M(\om))\in\D\big(\pi'\big(h_\om^{\frac{1}{2}}\big)\big)$. Moreover, we have
\[
 \s^{\pi'(\M)}(\pi'(h_\om))\Lambda(x)=\pi'(\s^\M(\om))\Lambda(x)=\Lambda(x\s^\M(\om)),
\]
which shows that $\Lambda(x\s^\M(\om))\in\Ha_0$. Consequently,
\[
 \Lambda(x\s^\M(\om))\in\D\big(\pi'\big(h_\om^{\frac{1}{2}}\big)\big|\Ha_0\big),
\]
and
\[
 \big(\pi'\big(h_\om^{\frac{1}{2}}\big)\big|\Ha_0\big)\Lambda(x\s^\M(\om))=\Lambda\big(x\s^\M(\om)h_\om^{\frac{1}{2}}\big)=\Lambda\big(xh_\om^{\frac{1}{2}}\big).
\]
This equality together with the equality \eqref{pi-1} yield
\[
 \Lambda\big(xh_\om^{\frac{1}{2}}\big)\in\D\big(\pi'\big(\widetilde{h_\om}^{\frac{1}{2}}\big)\big),
\]
and
\[
 \big(\pi'\big(\widetilde{h_\om}^{\frac{1}{2}}\big)\big)\Lambda\big(xh_\om^{\frac{1}{2}}\big)=\Lambda(x\s^\M(\om)).
\]
Since for $x\in\M\cap L^2(\M,\tau)$, we have $x\s^\M(\om)\in\M\cap L^2(\M,\tau)$, and thus $\Lambda(x\s^\M(\om))\in\D(\pi(h_\f^{\frac{1}{2}}))$, we get
\begin{equation}\label{bas}
 \begin{aligned}
  \Lambda\big(h_\f^{\frac{1}{2}}x\s^\M(\om)\big)&=\pi\big(h_\f^{\frac{1}{2}}\big)\Lambda(x\s^\M(\om))\\
  &=\pi\big(h_\f^{\frac{1}{2}}\big)\big(\pi'\big(\widetilde{h_\om}^{\frac{1}{2}}\big)\big)\Lambda(xh_\om^{\frac{1}{2}}).
 \end{aligned}
\end{equation}
Now we proceed as in the second part of the proof of Proposition \ref{Delta0}. For the positive cone $\{\Lambda(h^{\frac{1}{2}}):h\in L^1(\M,\tau), h\geq0\}$, we have
\begin{align*}
 &\langle\Lambda(h^{\frac{1}{2}})|\pi(h_\f^{\frac{1}{2}})\pi'\big(\widetilde{h_\om}^{\frac{1}{2}}\big)\Lambda(xh_\om^{\frac{1}{2}})\rangle_{\Ha}\\
 =&\langle\Lambda(h^{\frac{1}{2}})|\Lambda(h_\f^{\frac{1}{2}}x\s^\M(\om))\rangle_{\Ha}=\tau\big(h^{\frac{1}{2}}h_\f^{\frac{1}{2}}x\s^\M(\om)\big),
\end{align*}
and
\begin{align*}
 &\langle\Lambda(h^{\frac{1}{2}})|\Delta(\f,\om)^{\frac{1}{2}}\Lambda(xh_\om^{\frac{1}{2}})\rangle_{\Ha}
 =\overline{\langle J(\Lambda(h^{\frac{1}{2}}))|J\Delta(\f,\om)^{\frac{1}{2}}\Lambda(xh_\om^{\frac{1}{2}})\rangle}_{\Ha}\\
 =&\overline{\langle\Lambda(h^{\frac{1}{2}})|S\big(\Lambda\big(xh_\om^{\frac{1}{2}}\big)\big)\rangle}_{\Ha}
 =\langle\Lambda\big(\s^\M(\om)x^*h_\f^{\frac{1}{2}}\big)|\Lambda(h^{\frac{1}{2}})\rangle_{\Ha}\\
 =&\tau\big(h_\f^{\frac{1}{2}}x\s^\M(\om)h^{\frac{1}{2}}\big)=\tau\big(h^{\frac{1}{2}}h_\f^{\frac{1}{2}}x\s^\M(\om)\big).
\end{align*}
Consequently, repeating the reasoning as in Proposition \ref{Delta0}, we get
\[
 \pi(h_\f^{\frac{1}{2}})\pi'\big(\widetilde{h_\om}^{\frac{1}{2}}\big)\Lambda(xh_\om^{\frac{1}{2}})=\Delta(\f,\om)^{\frac{1}{2}}\Lambda(xh_\om^{\frac{1}{2}})
\]
for $x\in\M\cap L^2(\M,\tau)$. For $\xi\in\big(\1-\s^{\pi(\M)'}\big(\Lambda\big(h_\om^{\frac{1}{2}}\big)\big)\big)\Ha$, we have $S\xi=0$, and since $\Delta(\f,\om)^{\frac{1}{2}}=J\overline{S}$, we obtain
\[
 \Delta(\f,\om)^{\frac{1}{2}}\xi=0
\]
for such $\xi$. The formulae \eqref{s} yield
\[
 \s^{\pi(\M)'}\big(\Lambda\big(h_\om^{\frac{1}{2}}\big)\big)=\s^{\pi'(\M)}\big(\Lambda\big(h_\om^{\frac{1}{2}}\big)\big)=\pi'(\s^\M(\om))=\pi'(e((0,\infty))),
\]
and from the definition of $\pi'\big(\widetilde{h_\om}^{\frac{1}{2}}\big)$, it follows that
\[
 \pi'\big(\widetilde{h_\om}^{\frac{1}{2}}\big)\xi=0
\]
for $\xi\in\big(\1-\s^{\pi(\M)'}\big(\Lambda\big(h_\om^{\frac{1}{2}}\big)\big)\big)\Ha$. Our considerations can be summarised as follows
\begin{proposition'2}
Let $\widetilde{h_\om}$ be defined by the formula \eqref{h_om}. Then
\[
 \pi(h_\f^{\frac{1}{2}})\pi'\big(\widetilde{h_\om}^{\frac{1}{2}}\big)\big|\D(S_0)=\Delta(\f,\om)^{\frac{1}{2}}|\D(S_0).
\]
\end{proposition'2}
Now the remaining part of Section \ref{MO} can be rewritten word for word, the only change being in the replacement of the operator\\ $\pi'\big(h_\om^{-\frac{1}{2}}\big)$ by the operator $\pi'\big(\widetilde{h_\om}^{\frac{1}{2}}\big)$ (note that if $h_\om$ is invertible, i.e. if $\om$ is faithful, then we have simply $\widetilde{h_\om}^{\frac{1}{2}}=h_\om^{-\frac{1}{2}}$). In particular, for
\begin{equation}\tag{11'}
 \pi\big(h_\f^{\frac{1}{2}}\big)=\int_{-\infty}^\infty u(\lambda)\,m(d\lambda), \quad \pi'\Big(\widetilde{h_\om}^{\frac{1}{2}}\Big)=\int_{-\infty}^\infty v(\lambda)\,m(d\lambda),
\end{equation}
we have
\begin{theorem'}
Let $\pi(h_\f^{\frac{1}{2}})$ and $\pi'\Big(\widetilde{h_\om}^{\frac{1}{2}}\Big)$ be represented by the formula (11'). Then
\begin{equation}\tag{12'}
 \Delta(\f,\om)^{\frac{1}{2}}=\overline{\pi\big(h_\f^{\frac{1}{2}}\big)\pi'\Big(\widetilde{h_\om}^{\frac{1}{2}}\Big)}=\int_{-\infty}^\infty u(\lambda)v(\lambda)\,m(d\lambda).
\end{equation}
\end{theorem'}
From the considerations above, we obtain the following result.
\begin{lemma}\label{Delta-pi}
For arbitrary $s>0$, the following formula holds
\[
 \pi\big(h_\f^s\big)\pi'\Big(\widetilde{h_\om}^s\Big)\subset\Delta(\f,\om)^s.
\]
\end{lemma}
\begin{proof}
Indeed, from the formulae (11') and (12') we get
\[
 \pi\big(h_\f^s\big)=\int_{-\infty}^\infty u^{2s}(\lambda)\,m(d\lambda), \quad \pi'\Big(\widetilde{h_\om}^s\Big)=\int_{-\infty}^\infty v^{2s}(\lambda)\,m(d\lambda),
\]
and
\[
 \Delta(\f,\om)^s=\int_{-\infty}^\infty u^{2s}(\lambda)v^{2s}(\lambda)\,m(d\lambda),
\]
thus
\[
 \pi\big(h_\f^s\big)\pi'\Big(\widetilde{h_\om}^s\Big)\subset\int_{-\infty}^\infty u^{2s}(\lambda)v^{2s}(\lambda)\,m(d\lambda)=\Delta(\f,\om)^s. \qedhere
\]
\end{proof}
\begin{lemma}\label{p}
There are projections $e_n\in\M$ such that $e_n\to\1$ \emph{strongly}, $\tau(e_n)<+\infty$, and $e_nh_\f=h_\f e_n$.
\end{lemma}
\begin{proof}
Let
\[
 h_\f=\int_0^\infty\lambda\,e(d\lambda)
\]
be the spectral decomposition of $h_\f$. We have
\[
 h_\f\geq\int_{1/n}^\infty\lambda\,e(d\lambda)\geq\frac{1}{n}e\Big(\Big[\frac{1}{n}<+\infty\Big)\Big),
\]
and since $\tau(h_\f)<+\infty$, we infer that $\tau\big(e\big(\big[\frac{1}{n},+\infty\big)\big)\big)<+\infty$. Moreover, $e\big(\big[\frac{1}{n},+\infty\big)\big)\to e((0,+\infty))=\s(h_\f)$. For the null projection $e(\{0\})$, we choose projections $p_n\in\M$ such that $p_n\leq e(\{0\})$,\\ $p_n\to e(\{0\})$, and $\tau(p_n)<+\infty$ (of course, if $\tau(e(\{0\})<+\infty$, we can take $p_n=e(\{0\})$). Putting
\[
 e_n=e\Big(\Big[\frac{1}{n},+\infty\Big)\Big)+p_n,
\]
we obtain the conclusion.
\end{proof}
Recall that for a selfadjoint positive operator $h$, we have
\[
 h^0=\s(h),
\]
where $\s(h)$ is the support of $h$.

The following result is crucial in our further considerations.
\begin{proposition}
For arbitrary $0\leq s\leq\frac{1}{2}$, the following equality holds
\begin{equation}\label{basic}
 \Delta(\f,\om)^s\Lambda\big(h_\om^{\frac{1}{2}}\big)=\Lambda\big(h_\f^sh_\om^{\frac{1}{2}-s}\big).
\end{equation}
\end{proposition}
\begin{proof}
(Note that since $\Lambda\big(h_\om^{\frac{1}{2}}\big)\in\D(\Delta(\f,\om)^{\frac{1}{2}})$ and $s\leq\frac{1}{2}$, we have $\Lambda\big(h_\om^{\frac{1}{2}}\big)\in\D(\Delta(\f,\om)^s)$.)

Assume first that $0<s<\frac{1}{2}$.

Observe that for every $0\leq r\leq\frac{1}{2}$ and every $x\in\M\cap L^2(\M,\tau)$, we have $xh_\om^r\in L^2(\M,\tau)$. Indeed, let
\[
 h_\om=\int_0^\infty\lambda\,e(d\lambda)
\]
be the spectral decomposition of $h_\om$. Then
\begin{align*}
 &xh^r\big(xh^r\big)^*=xh^{2r}x^*=x\Big(\int_0^1\lambda^{2r}\,e(d\lambda)\Big)x^*+x\Big(\int_1^\infty\lambda^{2r}\,e(d\lambda)\Big)x^*\\
 \leq&xe([0,1]x^*+x\Big(\int_1^\infty\lambda\,e(d\lambda)\Big)x^*\leq xx^*+xh_\om x^*\in L^1(\M,\tau).
\end{align*}
Consequently, for every $x\in\M\cap L^2(\M,\tau)$, $\Lambda\big(xh_\om^{\frac{1}{2}}\big)$ belongs to the domain of $\pi'\big(\widetilde{h_\om}^s\big)$ and
\[
 \pi'\big(\widetilde{h_\om}^s\big)\Lambda\big(xh_\om^{\frac{1}{2}}\big)=\Lambda\big(xh_\om^{\frac{1}{2}-s}\big).
\]
Further, we have $h_\om^{\frac{1}{2}-s}\in L^{\frac{2}{1-2s}}(\M,\tau)$, thus $xh_\om^{\frac{1}{2}-s}\in L^{\frac{2}{1-2s}}(\M,\tau)$, and since $h_\f^s\in L^{\frac{1}{s}}$, we obtain that $h_\f^sxh_\om^{\frac{1}{2}-s}\in L^2(\M,\tau)$ which means that $\Lambda\big(xh_\om^{\frac{1}{2}-s}\big)$ belongs to the domain of $\pi\big(h_\f^s\big)$ and
\[
 \pi\big(h_\f^s\big)\Lambda\big(xh_\om^{\frac{1}{2}-s}\big)=\Lambda\big(h_\f^sxh_\om^{\frac{1}{2}-s}\big).
\]
From the considerations above and Lemma \ref{Delta-pi}, we obtain that for every $x\in\M\cap L^2(\M,\tau)$,
\begin{equation}\label{Delta2}
  \Delta(\f,\om)^s\Lambda\big(xh_\om^{\frac{1}{2}}\big)=\pi\big(h_\f^s\big)\pi'\big(\widetilde{h_\om}^s\big)\Lambda\big(xh_\om^{\frac{1}{2}}\big)
  =\Lambda\big(h_\f^sxh_\om^{\frac{1}{2}-s}\big).
\end{equation}
Choose projections $e_n\in\M$ as in Lemma \ref{p}, i.e. such that $e_n\to\1$ \emph{strongly}, $\tau(e_n)<+\infty$, and $e_nh_\f=h_\f e_n$. We have
\[
 \|\Lambda\big(e_nh_\om^{\frac{1}{2}}\big)-\Lambda\big(h_\om^{\frac{1}{2}}\big)\|_\Ha^2=\tau\big(h_\om^{\frac{1}{2}}(\1-e_n)h_\om^{\frac{1}{2}}\big)=\om(\1-e_n)\to0,
\]
which means that
\[
 \Lambda\big(e_nh_\om^{\frac{1}{2}}\big)\to\Lambda\big(h_\om^{\frac{1}{2}}\big) \quad \textit{in norm}
\]
(of course, here we have used only the convergence $e_n\to\1$), and
\begin{align*}
 &\|\Lambda\big(h_\f^se_nh_\om^{\frac{1}{2}-s}\big)-\Lambda\big(h_\f^sh_\om^{\frac{1}{2}-s}\big)\|_\Ha^2=\|\Lambda\big(h_\f^s(e_n-\1)h_\om^{\frac{1}{2}-s}\|_\Ha^2\\
 =&\tau\big(h_\om^{\frac{1}{2}-s}(e_n-\1)h_\f^{2s}(e_n-\1)h_\om^{\frac{1}{2}-s}\big)=\tau\big(h_\om^{\frac{1}{2}-s}h_\f^s(\1-e_n)h_\f^sh_\om^{\frac{1}{2}-s}\big)\\
 =&\tau\big(h_\f^sh_\om^{1-2s}h_\f^s(\1-e_n)\big),
\end{align*}
where the change of order under the trace in the last equality follows from the fact that $h_\f^sh_\om^{\frac{1}{2}-s}, h_\om^{\frac{1}{2}-s}h_\f^s(\1-e_n)\in L^2(\M,\tau)$. Since $h_\f^sh_\om^{1-2s}h_\f^s\in L^1(\M,\tau)$, there is $\ro\in\M_*$ such that for every $x\in\M$ we have
\[
 \tau\big(h_\f^sh_\om^{1-2s}h_\f^sx\big)=\ro(x).
\]
Consequently, we obtain
\begin{align*}
 \|\Lambda\big(h_\f^se_nh_\om^{\frac{1}{2}-s}\big)-\Lambda\big(h_\f^sh_\om^{\frac{1}{2}-s}\big)\|_\Ha^2&=\tau\big(h_\f^sh_\om^{1-2s}h_\f^s(\1-e_n)\big)\\
 &=\ro(\1-e_n)\to0,
\end{align*}
which means that
\[
 \Lambda\big(h_\f^se_nh_\om^{\frac{1}{2}-s}\big)\to\Lambda\big(h_\f^sh_\om^{\frac{1}{2}-s}\big) \quad \textit{in norm}.
\]
Further, we have
\[
 \D\big(\Delta(\f,\om)^s\big)\ni\Lambda\big(e_nh_\om^{\frac{1}{2}}\big)\to\Lambda\big(h_\om^{\frac{1}{2}}\big),
\]
and, putting $x=e_n$ in the formula \eqref{Delta2}, we get
\[
 \Delta(\f,\om)^s\Lambda\big(e_nh_\om^{\frac{1}{2}}\big)=\Lambda\big(h_\f^se_nh_\om^{\frac{1}{2}-s}\big)\to\Lambda\big(h_\f^sh_\om^{\frac{1}{2}-s}\big).
\]
Since $\Delta(\f,\om)^s$ is closed, it follows that the formula \eqref{basic} holds.

Now for $s=\frac{1}{2}$, the formula \eqref{basic} follows directly from the formula \eqref{bas}.

For $s=0$, we have on account of \cite[Theorem 2.4]{A1} and the formulae \eqref{s}, \eqref{s1}
\begin{align*}
 \Delta(\f,\om)^0=\s(\Delta(\f,\om))&=\s^{\pi(\M)'}\big(\Lambda\big(h_\om^{\frac{1}{2}}\big)\big)\s^{\pi(M)}\big(\Lambda\big(h_\f^{\frac{1}{2}}\big)\big)\\
 &=\pi'\big(\s^\M(\om)\big)\pi\big(\s^\M(\f)\big),
\end{align*}
and thus
\begin{align*}
 \Delta(\f,\om)^0\Lambda\big(h_\om^{\frac{1}{2}}\big)&=\pi'\big(\s^\M(\om)\big)\pi\big(\s^\M(\f)\big)\Lambda\big(h_\om^{\frac{1}{2}}\big)\\
 &=\Lambda\big(\s^\M(\f)h_\om^{\frac{1}{2}}\s^\M(\om)\big)=\Lambda\big(\s^\M(\f)h_\om^{\frac{1}{2}}\big)=\Lambda\big(h_\f^0h_\om^{\frac{1}{2}}\big),
\end{align*}
which finishes the proof.
\end{proof}
Observe that as a direct consequence of the proposition above, we obtain Theorem \ref{tau} which allows us to change order under the trace for arbitrary elements belonging to $L^p(\M,\tau)$ and $L^q(\M,\tau)$, respectively. This will be commented on in Section \ref{RE-I} where some applications of our results about the relative modular operator are presented. At the moment, we shall exploit this property in order to obtain a more general version of the formula \eqref{basic}. We start with the following technical lemma.
\begin{lemma}\label{est}
Let $z\in\M$. The following estimates hold true:
\begin{enumerate}
 \item[(i)] for $0<s\leq\frac{1}{4}$
  \[
   \tau\big(h_\om^{\frac{1}{2}-s}z^*h_\f^{2s}zh_\om^{\frac{1}{2}-s}\big)\leq\|h_\f\|_1^{2s}\|h_\om\|_1^{\frac{1}{2}-2s}\|z\|_\infty\|zh_\om^{\frac{1}{2}}\|_2,
  \]
 \item[(ii)] for $\frac{1}{4}<s\leq\frac{1}{2}$
  \[
   \tau\big(h_\om^{\frac{1}{2}-s}z^*h_\f^{2s}zh_\om^{\frac{1}{2}-s}\big)\leq\|h_\f\|_1^{2s-\frac{1}{2}}\|h_\om\|_1^{1-2s}\|z\|_\infty\|h_\f^{\frac{1}{2}}z\|_2.
  \]
\end{enumerate}
\end{lemma}
\begin{proof}
We have $h_\om^{\frac{1}{2}-s}z^*\in L^{\frac{2}{1-2s}}(\M,\tau)$, $h_\f^{2s}zh_\om^{\frac{1}{2}-s}\in L^{\frac{2}{1+2s}}(\M,\tau)$, hence
\begin{equation}\label{h1}
 \begin{aligned}
  &\tau\big(h_\om^{\frac{1}{2}-s}z^*h_\f^{2s}zh_\om^{\frac{1}{2}-s}\big)=\tau\big(h_\f^{2s}zh_\om^{1-2s}z^*\big)\\
  \leq&\|z\|_\infty\tau\big(|h_\f^{2s}zh_\om^{1-2s}|\big)
  =\|z\|_\infty\tau\big(u^*h_\f^{2s}zh_\om^{1-2s}\big),
 \end{aligned}
\end{equation}
where
\[
 h_\f^{2s}zh_\om^{1-2s}=u|h_\f^{2s}zh_\om^{1-2s}|
\]
is the polar decomposition.

(i) Let $0<s<\frac{1}{4}$. Then
\[
 \tau\big(u^*h_\f^{2s}zh_\om^{1-2s}\big)=\tau\big(u^*h_\f^{2s}zh_\om^{\frac{1}{2}}h_\om^{\frac{1}{2}-2s}\big)=\tau\big(h_\om^{\frac{1}{2}-2s}u^*h_\f^{2s}zh_\om^{\frac{1}{2}}\big).
\]
Since $h_\om^{\frac{1}{2}-2s}u^*h_\f^{2s},zh_\om^{\frac{1}{2}}\in L^2(\M,\tau)$, the H\"{o}lder inequality yields
\begin{equation}\label{h2}
 \tau\big(h_\om^{\frac{1}{2}-2s}u^*h_\f^{2s}zh_\om^{\frac{1}{2}}\big)\leq\|h_\om^{\frac{1}{2}-2s}u^*h_\f^{2s}\|_2\|zh_\om^{\frac{1}{2}}\|_2.
\end{equation}
Further, we have $h_\om^{\frac{1}{2}-2s}\in L^{\frac{2}{1-4s}}(\M,\tau)$, $u^*h_\f^{2s}\in L^{\frac{1}{2s}}(\M,\tau)$, thus using again the H\"{o}lder inequality with $p=\frac{2}{1-4s}$, $q=\frac{1}{2s}$, and $r=2$, we obtain
\begin{equation}\label{h3}
 \begin{aligned}
  \|h_\om^{\frac{1}{2}-2s}u^*h_\f^{2s}\|_2&\leq\|h_\om^{\frac{1}{2}-2s}\|_{\frac{2}{1-4s}}\|u^*h_\f^{2s}\|_{\frac{1}{2s}}\\
  &\leq\|h_\om^{\frac{1}{2}-2s}\|_{\frac{2}{1-4s}}\|u^*\|_\infty\|h_\f^{2s}\|_{\frac{1}{2s}}\\
  &\leq\|h_\om^{\frac{1}{2}-2s}\|_{\frac{2}{1-4s}}\|h_\f^{2s}\|_{\frac{1}{2s}}=\|h_\om\|_1^{\frac{1}{2}-2s}\|h_\f\|_1^{2s}.
 \end{aligned}
\end{equation}
Now inserting the estimates \eqref{h2} and \eqref{h3} into the inequality \eqref{h1}, we obtain the claim.

For $s=\frac{1}{4}$, we get
\begin{align*}
 &\tau\big(u^*h_\f^{2s}zh_\om^{1-2s}\big)=\tau\big(u^*h_\f^{\frac{1}{2}}zh_\om^{\frac{1}{2}}\big)\leq\|u^*h_\f^{\frac{1}{2}}\|_2\|zh_\om^{\frac{1}{2}}\|_2\\
 &\leq\|u^*\|_\infty\|h_\f^{\frac{1}{2}}\|_2\|zh_\om^{\frac{1}{2}}\|_2\leq\|h_\f\|_1^{\frac{1}{2}}\|zh_\om^{\frac{1}{2}}\|_2,
\end{align*}
which again shows the claim.

(ii) Let now $\frac{1}{4}<s\leq\frac{1}{2}$. Then
\[
 \tau\big(u^*h_\f^{2s}zh_\om^{1-2s}\big)=\tau\big(u^*h_\f^{2s-\frac{1}{2}}h_\f^{\frac{1}{2}}zh_\om^{1-2s}\big)
 =\tau\big(h_\om^{1-2s}u^*h_\f^{2s-\frac{1}{2}}h_\f^{\frac{1}{2}}z).
\]
Since $h_\om^{1-2s}u^*h_\f^{2s-\frac{1}{2}}, h_\f^{\frac{1}{2}}z\in L^2(\M,\tau)$, the H\"{o}lder inequality yields
\begin{equation}\label{h4}
 \tau\big(h_\om^{1-2s}u^*h_\f^{2s-\frac{1}{2}}h_\f^{\frac{1}{2}}z\big)\leq\|h_\om^{1-2s}u^*h_\f^{2s-\frac{1}{2}}\|_2\|h_\f^{\frac{1}{2}}z\|_2.
\end{equation}
Further, we have $h_\om^{1-2s}u^*\in L^{\frac{1}{1-2s}}(\M,\tau)$, $h_\f^{2s-\frac{1}{2}}\in L^{\frac{2}{4s-1}}(\M,\tau)$, thus using again the H\"{o}lder inequality with $p=\frac{1}{1-2s}$, $q=\frac{2}{4s-1}$, and $r=2$, we obtain
\begin{equation}\label{h5}
 \begin{aligned}
  &\|h_\om^{1-2s}u^*h_\f^{2s-\frac{1}{2}}\|_2\leq\|h_\om^{1-2s}u^*\|_{\frac{1}{1-2s}}\|h_\f^{2s-\frac{1}{2}}\|_{\frac{2}{4s-1}}\\
  \leq&\|h_\om^{1-2s}\|_{\frac{1}{1-2s}}\|u^*\|_\infty\|h_\f^{2s-\frac{1}{2}}\|_{\frac{2}{4s-1}}\\
  \leq&\|h_\om^{1-2s}\|_{\frac{1}{1-2s}}\|h_\f^{2s-\frac{1}{2}}\|_{\frac{2}{4s-1}}=\|h_\om\|_1^{1-2s}\|h_\f\|_1^{2s-\frac{1}{2}}.
 \end{aligned}
\end{equation}
Inserting the estimates \eqref{h4} and \eqref{h5} into the inequality \eqref{h1}, we obtain the claim.
\end{proof}
Denote
\[
 \mathcal{K}=\{\Lambda(xh_\om^{1/2}):x\in\M\}+\big(\1-\s^{\pi(\M)'}\big(\Lambda\big(h_\om^{1/2}\big)\big)\big)\Ha,
\]
and
\[
 \mathcal{K}_0=\{\Lambda\big(xh_\om^{1/2}\big):x\in \M\cap L^2(\M,\tau)\}+\big(\1-\s^{\pi(\M)'}\big)\Ha.
\]
(Observe that $\mathcal{K}$ and $\mathcal{K}_0$ were denoted in Chapter \ref{MO1} respectively as $\D(S)$ and $\D(S_0)$, however, since we shall not refer to the operators $S$ and $S_0$, we choose a more straightforward notation.)
\begin{proposition}
For every $0<s\leq\frac{1}{2}$, we have
\[
 \overline{\Delta(\f,\om,)^s|\mathcal{K}_0}=\Delta(\f,\om)^s.
\]
Moreover, the following formula holds true
\begin{equation}\label{basic1}
 \Delta(\f,\om)^s\big(\Lambda\big(xh_\om^{\frac{1}{2}}\big)=\Lambda\big(h_\f^sxh_\om^{\frac{1}{2}-s}\big),\quad x\in\M.
\end{equation}
(As mentioned in the Introduction, this formula was obtained in the setup of Haagerup's $L^2$-spaces in \cite[Proposition 12.6]{Hi}.)
\end{proposition}
\begin{proof}
For arbitrary $x\in\M$, choose, according to Lemma \ref{sstrong},\\ $x_n\in\M\cap L^2(\M,\tau)$ such that $x_n\to x$ $\sigma$-\emph{strongly}*. We have
\begin{equation}\label{est1}
 \begin{aligned}
  &\|\Lambda\big(x_nh_\om^{\frac{1}{2}}\big)-\Lambda\big(xh_\om^{\frac{1}{2}}\big)\|_\Ha^2=\|\Lambda\big((x_n-x)h_\om^{\frac{1}{2}}\big)\|_\Ha^2\\
  &\langle\Lambda\big((x_n-x)h_\om^{\frac{1}{2}}\big)|\Lambda\big((x_n-x)h_\om^{\frac{1}{2}}\big)\rangle_\Ha\\
  =&\tau\big(h_\om^{\frac{1}{2}}(x_n-x)^*(x_n-x)h_\om^{\frac{1}{2}}\big)=\om((x_n-x)^*(x_n-x))\to0,
 \end{aligned}
\end{equation}
and
\begin{equation*}
 \begin{aligned}
  &\|\Lambda\big(h_\f^sx_nh_\om^{\frac{1}{2}-s}\big)-\Lambda\big(h_\f^sxh_\om^{\frac{1}{2}-s}\big)\|_\Ha^2=\|\Lambda\big(h_\f^s(x_n-x)h_\om^{\frac{1}{2}-s}\big)\|_\Ha^2\\
  =&\langle\Lambda\big(h_\f^s(x_n-x)h_\om^{\frac{1}{2}-s}\big)|\Lambda\big(h_\f^s(x_n-x)h_\om^{\frac{1}{2}-s}\big)\rangle_\Ha\\
  =&\tau\big(h_\om^{\frac{1}{2}-s}(x_n-x)^*h_\f^{2s}(x_n-x)h_\om^{\frac{1}{2}-s}\big)\\
  =&\tau\big(h_\om^{1-2s}(x_n-x)^*h_\f^{2s}(x_n-x)\big).
 \end{aligned}
\end{equation*}
Using Lemma \ref{est}, we obtain the estimates
\begin{equation}\label{est2}
 \begin{aligned}
  &\|\Lambda\big(h_\f^sx_nh_\om^{\frac{1}{2}-s}\big)-\Lambda\big(h_\f^sxh_\om^{\frac{1}{2}-s}\big)\|_\Ha^2\\
  =&\tau\big(h_\om^{1-2s}(x_n-x)^*h_\f^{2s}(x_n-x)\big)\\
  \leq&\|h_\f\|_1^{2s}\|h_\om\|_1^{\frac{1}{2}-2s}\|x_n-x\|_\infty\|(x_n-x)h_\om^{\frac{1}{2}}\|_2
 \end{aligned}
\end{equation}
for $0<s\leq\frac{1}{4}$, and
\begin{equation}\label{est3}
 \begin{aligned}
  &\|\Lambda\big(h_\f^sx_nh_\om^{\frac{1}{2}-s}\big)-\Lambda\big(h_\f^sxh_\om^{\frac{1}{2}-s}\big)\|_\Ha^2\\
  =&\tau\big(h_\om^{1-2s}(x_n-x)^*h_\f^{2s}(x_n-x)\big)\\
  \leq&\|h_\f\|_1^{2s-\frac{1}{2}}\|h_\om\|_1^{1-2s}\|x_n-x\|_\infty\|h_\f^{\frac{1}{2}}(x_n-x)\|_2,
 \end{aligned}
\end{equation}
for $\frac{1}{4}<s<\frac{1}{2}$. Now we have
\begin{equation}\label{est4}
 \begin{aligned}
  \|(x_n-x)h_\om^{\frac{1}{2}}\|_2^2&=\tau\big(h_\om^{\frac{1}{2}}(x_n-x)^*(x_n-x)h_\om^{\frac{1}{2}}\big)\\
  &=\om((x_n-x)^*(x_n-x))\to0,
 \end{aligned}
\end{equation}
and
\begin{equation}\label{est5}
 \begin{aligned}
  &\|h_\f^{\frac{1}{2}}(x_n-x)\|_2^2=\tau\big((x_n-x)^*h_\f(x_n-x)\big)\\
  =&\tau\big(h_\f^{\frac{1}{2}}(x_n-x)(x_n-x)^*h_\f^{\frac{1}{2}}\big)=\f((x_n-x)(x_n-x)^*)\to0.
 \end{aligned}
\end{equation}
Since $x_n\to x$ $\sigma$-\emph{strongly}*, the norms $\|x_n-x\|_\infty$ are bounded, consequently, the relations \eqref{est2}, \eqref{est3}, \eqref{est4}, and \eqref{est5} yield
\[
 \Lambda\big(h_\f^sx_nh_\om^{\frac{1}{2}-s}\big)\to\Lambda\big(h_\f^sxh_\om^{\frac{1}{2}-s}\big) \quad \textit{in norm}.
\]
The relation \eqref{est1} yields
\begin{equation}\label{close1}
 \Lambda\big(x_nh_\om^{\frac{1}{2}}\big)\to\Lambda\big(xh_\om^{\frac{1}{2}}\big) \quad \textit{in norm}.
\end{equation}
For $\xi\in\big(\1-\s^{\pi(\M)'}\big(\Lambda\big(h_\om^{\frac{1}{2}}\big)\big)\big)\Ha$ we have
\[
 \pi'\big(\widetilde{h_\om}^{\frac{1}{2}}\big)\xi=0,
\]
and thus
\[
 \pi'\big(\widetilde{h_\om}^s\big)\xi=0,
\]
consequently, since $\pi\big(h_\f^s\big)\pi'\big(\widetilde{h_\om}^s\big)\subset\Delta(\f,\om)^s$, we get
\[
 \Delta(\f,\om)^s\xi=0.
\]
For $\xi$ as above, we have, on account of the relation \eqref{close1},
\[
 \mathcal{K}_0\ni\Lambda\big(x_nh_\om^{\frac{1}{2}}\big)+\xi\to\Lambda\big(xh_\om^{\frac{1}{2}}\big)+\xi\in\mathcal{K} \quad \textit{in norm}.
\]
Further, we have
\begin{equation}\label{basic''}
 \begin{aligned}
  &\Delta(\f,\om)^s\big(\Lambda\big(x_nh_\om^{\frac{1}{2}}\big)+\xi\big)
  =\pi\big(h_\f^s\big)\pi'\big(\widetilde{h_\om}^s\big)\big(\Lambda\big(x_nh_\om^{\frac{1}{2}}\big)+\xi\big)\\
  =&\pi\big(h_\f^s\big)\pi'\big(\widetilde{h_\om}^s\big)\Lambda\big(x_nh_\om^{\frac{1}{2}}\big)
  =\Lambda\big(h_\f^sx_nh_\om^{\frac{1}{2}-s}\big)\to\Lambda\big(h_\f^sxh_\om^{\frac{1}{2}-s}\big),
 \end{aligned}
\end{equation}
which yields that $\Lambda\big(xh_\om^{\frac{1}{2}}\big)+\xi\in\D(\overline{\Delta(\f,\om)^s|\mathcal{K}_0})$, i.e.
\[
 \mathcal{K}\subset\D(\overline{\Delta(\f,\om)^s|\mathcal{K}_0}).
\]
It follows that
\[
 \Delta(\f,\om)^s|\mathcal{K}\subset\overline{\Delta(\f,\om)^s|\mathcal{K}_0},
\]
which, since $\mathcal{K}_0\subset\mathcal{K}$, yields the equality
\[
 \overline{\Delta(\f,\om)^s|\mathcal{K}}=\overline{\Delta(\f,\om)^s|\mathcal{K}_0}.
\]
Since $\mathcal{K}$ is a core for $\Delta(\f,\om)^s$, we get
\[
 \Delta(\f,\om)^s=\overline{\Delta(\f,\om)^s|\mathcal{K}}=\overline{\Delta(\f,\om)^s|\mathcal{K}_0}.
\]
Moreover, the relations \eqref{close1}, \eqref{basic''}, and the closedness of $\Delta(\f,\om)^s$ yield the formula \eqref{basic1}.
\end{proof}
\begin{theorem}
For $0<s\leq\frac{1}{2}$, the following formula holds
\begin{equation}\label{basic2}
 \Delta(\f,\om)^s=\overline{\pi\big(h_\f^s\big)\pi'\big(\widetilde{h_\om}^s\big)}.
\end{equation}
\end{theorem}
\begin{proof}
We have
\[
 \Delta(\f,\om)^s|\mathcal{K}_0=\pi\big(h_\f^s\big)\pi'\big(\widetilde{h_\om}^s\big)|\mathcal{K}_0\subset\pi\big(h_\f^s\big)\pi'\big(\widetilde{h_\om}^s\big),
\]
thus
\[
 \Delta(\f,\om)^s=\overline{\Delta(\f,\om)^s|\mathcal{K}_0}\subset\overline{\pi\big(h_\f^s\big)\pi'\big(\widetilde{h_\om}^s\big)}.
\]
The reverse inclusion follows from Lemma \ref{Delta-pi}.
\end{proof}

\section{Trace formula, relative entropy and information between states, quasi-entropies, and R\'enyi's relative entropy}\label{RE-I}
In this section, we employ the results on the relative modular operator to prove some basic property of a trace, to prove equality between Araki's relative entropy and information between states as defined by Umegaki, and to obtain formulae for quasi entropies and R\'enyi's relative entropy.

\subsection{Trace formula}\label{trform}
The result presented below was obtained in\\ \cite[Theorem 17]{BK} in a more general setting assuming only that\\ $xy,yx\in L^1(\M,\tau)$. Nevertheless, it seems interesting to see how the relative modular operator can be used to get some basic formula for the trace.

One of the defining properties of the trace is the equality
\[
 \tau(x^*x)=\tau(xx^*),
\]
which, after using a polarization formula, yields the equality
\[
 \tau(xy)=\tau(yx)
\]
for arbitrary $x,y\in L^2(\M,\tau)$. Our aim is to obtain this formula in a more general case.
\begin{theorem}\label{tau}
Let $x\in L^p(\M,\tau)$, $y\in L^q(\M,\tau)$, $\frac{1}{p}+\frac{1}{q}=1$. Then
\[
 \tau(xy)=\tau(yx).
\]
\end{theorem}
\begin{proof}
Let $h_\f$ and $h_\om$ be arbitrary density operators corresponding to normal states $\f$ and $\om$. For arbitrary $0<s\leq\frac{1}{2}$, we obtain, on account of the formula \eqref{basic},
\begin{equation}\label{basic'}
 \begin{aligned}
  &\langle\Lambda\big(h_\om^{\frac{1}{2}}\big)|\Delta(\f,\om)^s\Lambda\big(h_\om^{\frac{1}{2}}\big)\rangle_\Ha
  =\langle\Lambda\big(h_\om^{\frac{1}{2}}\big)|\Lambda\big(h_\f^sh_\om^{\frac{1}{2}-s}\big)\rangle_\Ha\\
  =&\tau\big(h_\om^{\frac{1}{2}}h_\f^sh_\om^{\frac{1}{2}-s}\big)=\tau\big(h_\f^sh_\om^{1-s}\big),
 \end{aligned}
\end{equation}
since $h_\om^{\frac{1}{2}},h_\f^sh_\om^{\frac{1}{2}-s}\in L^2(\M,\tau)$. Since $\Delta(\f,\om)^s$ is positive, it follows that on the left hand side of the formula \eqref{basic'} we have a nonnegative number, consequently,
\begin{align*}
 0\leq\langle\Lambda\big(h_\om^{\frac{1}{2}}\big)|\Delta(\f,\om)^s\Lambda\big(h_\om^{\frac{1}{2}}\big)\rangle_\Ha&=\tau\big(h_\f^sh_\om^{1-s}\big)=\overline{\tau\big(h_\f^sh_\om^{1-s}\big)}\\
 &=\tau\big(\big(h_\f^sh_\om^{1-s}\big)^*\big)=\tau\big(h_\om^{1-s}h_\f^s\big),
\end{align*}
i.e.
\begin{equation}\label{t}
 \tau\big(h_\f^sh_\om^{1-s}\big)=\tau\big(h_\om^{1-s}h_\f^s\big).
\end{equation}

Let now $0\leq x\in L^p(\M,\tau)$, $0\leq y\in L^q(\M,\tau)$. We have either $p\geq2$ or $q\geq2$, so assume that $p\geq2$. Put $s=\frac{1}{p}$, and define $h_\f$ and $h_\om$ by the formulae
\[
 h_\f=x^{\frac{1}{s}}=x^p, \qquad h_\om=y^{\frac{1}{1-s}}=y^q.
\]
Then $h_\f,h_\om$ are density operators, and the equality \eqref{t} yields
\[
 \tau(xy)=\tau\big(h_\f^sh_\om^{1-s}\big)=\tau\big(h_\om^{1-s}h_\f^s\big)=\tau(yx).
\]
For selfadjoint $x\in L^p(\M,\tau)$, $y\in L^q(\M,\tau)$, we have the Jordan decompositions
\[
 x=x^+-x^-, \qquad y=y^+-y^-,
\]
and since
\[
 |x|=x^++x^-\geq x^\pm. \qquad |y|=y^++y^-\geq y^\pm,
\]
we see that $x^\pm\in L^p(\M,\tau)$ and $y^\pm\in L^q(\M,\tau)$. Consequently,\\ $x^\pm y^\pm\in L^1(\M,\tau)$, and from the first part, we obtain
\[
 \tau(x^\pm y^\pm)=\tau(y^\pm x^\pm),
\]
hence
\begin{align*}
 \tau(xy)&=\tau((x^+-x^-)(y^+-y^-))\\
 &=\tau(x^+y^+-x^+y^--x^-y^++x^-y^-)\\
 &=\tau(y^+x^+-y^-x^+-y^+x^-+y^-x^-)\\
 &=\tau((y^+-y^-)(x^+-x^-))=\tau(yx).
\end{align*}
For arbitrary $x\in L^p(\M,\tau)$, $y\in L^q(\M,\tau)$, we obtain the result decomposing $x$ and $y$ as $x=x_1+ix_2$, $y=y_1+iy_2$ with $x_1$, $x_2$, $y_1$, $y_2$ selfadjoint.
\end{proof}

\subsection{Relative entropy and information between states}\label{ss1}
Let $\N$ be an arbitrary von Neumann algebra, and let $\f,\om\in\N_*^+$ be (non-normalised) states. Assume that $\s^\N(\om)\leq\s^\N(\f)$. Represent $\N$ such that in this representation $\om$ and $\f$ are vector states given respectively by vectors $\xi$ and $\eta$ in the space of the representation:
\[
 \om(\cdot)=\langle\xi|\cdot\xi\rangle, \qquad \f(\cdot)=\langle\eta|\cdot\eta\rangle.
\]
The \emph{relative entropy} $S(\f,\om)$ is defined as
\[
 S(\f,\om)=-\langle\xi|\log\Delta(\f,\omega)\xi\rangle,
\]
where $\Delta(\f,\om)$ is the relative modular operator. Moreover, $S(\f,\om)$ is independent of the representation chosen. Note that the form of the definition above is a little formal since $\xi$ may not belong to the domain of $\log\Delta(\f,\om)$. It is to be understood as
\[
 S(\f,\om)=-\int_0^\infty\log\lambda\,\langle\xi|e(d\lambda)\xi\rangle=-\int_0^\infty\log\lambda\,\|e(d\lambda)\xi\|^2,
\]
where
\[
 \Delta(\f,\om)=\int_0^\infty\lambda\,e(d\lambda)
\]
is the spectral decomposition of $\Delta(\f,\om)$.

In our case, we have, taking into account the equalities (11') and (12'),
\begin{equation}\label{S}
 S(\f,\om)=-2\int_{-\infty}^\infty\log(u(\lambda)v(\lambda))\,\|m(d\lambda)\Lambda\big(h_\om^{\frac{1}{2}}\big)\|_{\Ha}^2.
\end{equation}
For the states $\f$ and $\om$, the \emph{information} $I(\om,\f)$ between these states is defined in \cite{U} by the formula
\[
 I(\om,\f)=\tau(h_\om\log h_\om-h_\om\log h_\f),
\]
under the assumption that $\s^\M(\om)\leq\s^\M(\f)$, and that the \emph{Segal entropy} $H(\om)$ of $\om$ defined as
\[
 H(\om)=\tau(h_\om\log h_\om)
\]
is finite. (In the original Segal definition \cite{S1}, there is a minus sign before the trace; we choose the version as above for simplicity and in order that $H(\om)$ be nonnegative for a normalised state and finite trace.) The point in the requirement $\s^\M(\om)\leq\s^\M(\f)$ is that
\[
 \s^\M(\om)=\s^\M(h_\om), \qquad \s^\M(\f)=\s^\M(h_\f),
\]
so if for a vector $\xi$ we have $h_\f\xi=0$, then $h_\om\xi=0$, in which case we define $(h_\om\log h_\f)\xi=0$. It is known that in the finite dimensional case we have
\[
 S(\f,\om)=I(\om,\f).
\]
However, in infinite dimension serious problems arise. First, despite the statement in \cite[p. 69]{U}, the operator $h_\om\log h_\f$ \emph{need not be measurable if $\M$ is not finite} because $\log h_\f$ need not be measurable (by the way, a similar mistake is made further in the proof of Proposition 4.1 where it is stated that the operator $(b+p)^{-1}$ is measurable --- again, if e.g. $p=0$ which amounts to the fact that $b$ is invertible, then $b^{-1}$ need not be measurable). Consequently, it may happen that the domain of $h_\om\log h_\f$ is $\{0\}$ which makes the whole definition of $\tau(h_\om\log h_\f)$ questionable. A natural way out seems to be as follows. The expression $\tau(h_\om\log h_\f)$ can formally be regarded as $\om(\log h_\f)$ which in turn requires a reasonable definition of the objects like $\om(x)$ for unbounded $x$. To this end, assume first that $x$ is a selfadjoint positive operator affiliated with $\M$, with the spectral decomposition
\[
 x=\int_0^\infty\lambda\,e(d\lambda),
\]
and define
\begin{equation}\label{om}
 \om(x)=\int_0^\infty\lambda\,\om(e(d\lambda)).
\end{equation}
(Observe that for $x\in\M$ we do have the formula \eqref{om}.) To justify this definition in general, observe that the following lemma holds true.
\begin{lemma}\label{om-tau}
Let $x$ be selfadjoint positive and measurable. Then $\om(x)$ is finite if and only if $h_\om^{\frac{1}{2}}xh_\om^{\frac{1}{2}}\in L^1(\M,\tau)$, in which case
\[
 \om(x)=\tau(h_\om^{\frac{1}{2}}xh_\om^{\frac{1}{2}}).
\]
\end{lemma}
\begin{proof}
Let $x_{[n]}$ be the truncation of $x$ defined by the formula \eqref{trunc}, i.e.
\[
 x_{[n]}=\int_0^n\lambda\,e(d\lambda).
\]
Observe first that we have
\[
 \lim_{n\to\infty}\om(x_{[n]})=\lim_{n\to\infty}\om\Big(\int_0^n\lambda\,e(d\lambda)\Big)=\lim_{n\to\infty}\int_0^n\lambda\,\om(e(d\lambda))=\om(x).
\]
Let $\varepsilon>0$ be arbitrary. From the measurability of $x$, it follows that there is $\lambda_0$ such that $\tau(e((\lambda_0,+\infty)))<\varepsilon$, and for $n>\lambda_0$ we have
\[
 (x_{[n]}-x)e([0,\lambda_0])=0,
\]
which means that $x_{[n]}\to x$ \emph{in Segal's sense}.

Assume that $\om(x)<+\infty$. For $n\geq m$ we have, taking into account the inequality $h_\om^{\frac{1}{2}}x_nh_\om^{\frac{1}{2}}-h_\om^{\frac{1}{2}}x_mh_\om^{\frac{1}{2}}\geq0$,
\begin{align*}
 \|h_\om^{\frac{1}{2}}x_{[n]}h_\om^{\frac{1}{2}}-h_\om^{\frac{1}{2}}x_{[m]}h_\om^{\frac{1}{2}}\|_1&=\tau\big(h_\om^{\frac{1}{2}}(x_{[n]}-x_{[m]})h_\om^{\frac{1}{2}}\big)\\
 &=\om(x_{[n]}-x_{[m]})=\int_m^n\lambda\,\om(e(d\lambda))\underset{n,m\to\infty}{\longrightarrow0},
\end{align*}
which means that the sequence $\big(h_\om^{\frac{1}{2}}x_{[n]}h_\om^{\frac{1}{2}}\big)$ is Cauchy in $\|\cdot\|_1$-norm. Consequently,
\[
 h_\om^{\frac{1}{2}}x_{[n]}h_\om^{\frac{1}{2}}\to z \quad\text{\emph{in $\|\cdot\|_1$-norm}}
\]
for some $z\in L^1(\M,\tau)$. On the other hand, we have $x_{[n]}\to x$ \emph{in Segal's sense}, thus $x_{[n]}\to x$ \emph{in measure}, hence $h_\om^{\frac{1}{2}}x_{[n]}h_\om^{\frac{1}{2}}\to h_\om^{\frac{1}{2}}xh_\om^{\frac{1}{2}}$ \emph{in measure}. Since convergence in $\|\cdot\|_1$-norm implies convergence in measure, we get $h_\om^{\frac{1}{2}}xh_\om^{\frac{1}{2}}=z\in L^1(\M,\tau)$, and convergence in $\|\cdot\|_1$-norm implies
\[
 \tau\big(h_\om^{\frac{1}{2}}xh_\om^{\frac{1}{2}}\big)=\lim_{n\to\infty}\tau\big(h_\om^{\frac{1}{2}}x_{[n]}h_\om^{\frac{1}{2}}\big)=\lim_{n\to\infty}\om(x_{[n]})=\om(x).
\]

Assume now that $h_\om^{\frac{1}{2}}xh_\om^{\frac{1}{2}}\in L^1(\M,\tau)$. Since
\[
 h_\om^{\frac{1}{2}}x_{[n]}h_\om^{\frac{1}{2}}\leq h_\om^{\frac{1}{2}}xh_\om^{\frac{1}{2}},
\]
we obtain
\begin{align*}
 \om(x)=\lim_{n\to\infty}\om(x_{[n]})=\lim_{n\to\infty}\tau\big(h_\om^{\frac{1}{2}}x_{[n]}h_\om^{\frac{1}{2}}\big)\leq\tau\big(h_\om^{\frac{1}{2}}xh_\om^{\frac{1}{2}}),
\end{align*}
i.e. $\om(x)<+\infty$.
\end{proof}
\begin{remark}
Observe that $\om(x)$ defined by \eqref{om} always makes sense, as does $\tau\big(h_\om^{\frac{1}{2}}xh_\om^{\frac{1}{2}}\big)$, because $h_\om^{\frac{1}{2}}xh_\om^{\frac{1}{2}}\geq0$, and the lemma above yields the equality
\[
 \om(x)=\tau\big(h_\om^{\frac{1}{2}}xh_\om^{\frac{1}{2}}\big)
\]
also for either of the terms being equal to infinity.
\end{remark}
Now arbitrary selfadjoint $x$ affiliated with $\M$ has the Jordan decomposition
\[
 x=x^+-x^-,
\]
where
\[
 x=\int_{-\infty}^{+\infty}\lambda\,e(d\lambda), \qquad x^+=\int_0^\infty\lambda\,e(d\lambda), \qquad x^-=-\int_{-\infty}^0\lambda\,e(d\lambda),
\]
and we define
\[
 \om(x)=\om(x^+)-\om(x^-),
\]
whenever $\min\{\om(x^-),\om(x^+)\}<\infty$. Then $\om(x)$ is finite if and only if $\om(|x|)$ is finite. From Lemma \ref{om-tau}, we get
\begin{lemma}\label{om-tau1}
Let $x$ be selfadjoint and measurable. Then $\om(x)$ is finite if and only if $h_\om^{\frac{1}{2}}xh_\om^{\frac{1}{2}}\in L^1(\M,\tau)$, in which case
\begin{equation}\label{om-tau0}
 \om(x)=\tau(h_\om^{\frac{1}{2}}xh_\om^{\frac{1}{2}}).
\end{equation}
\end{lemma}
\begin{proof}
We have
\[
 \om(x)=\om(x^+)-\om(x^-)=\tau\big(h_\om^{\frac{1}{2}}x^+h_\om^{\frac{1}{2}}\big)-\tau\big(h_\om^{\frac{1}{2}}x^-h_\om^{\frac{1}{2}}\big),
\]
and the finiteness of $\om(x)$ is equivalent to the finiteness of $\om(x^+)$ and $\om(x^-)$ which in turn is equivalent to the finiteness of $\tau\big(h_\om^{\frac{1}{2}}x^+h_\om^{\frac{1}{2}}\big)$ and the finiteness of $\tau\big(h_\om^{\frac{1}{2}}x^-h_\om^{\frac{1}{2}}\big)$, i.e. $h_\om^{\frac{1}{2}}x^+h_\om^{\frac{1}{2}}\in L^1(\M,\tau)$ and\\ $h_\om^{\frac{1}{2}}x^-h_\om^{\frac{1}{2}}\in L^1(\M,\tau)$, so
\[
 h_\om^{\frac{1}{2}}xh_\om^{\frac{1}{2}}=h_\om^{\frac{1}{2}}x^+h_\om^{\frac{1}{2}}-h_\om^{\frac{1}{2}}x^-h_\om^{\frac{1}{2}}\in L^1(\M,\tau),
\]
and obviously the formula \eqref{om-tau0} follows.
\end{proof}
We have also the following relation.
\begin{lemma}
Let $x$ be selfadjoint and measurable, and assume that\\ $h_\om x\in L^1(\M,\tau)$. Then
\[
 \om(x)=\tau(h_\om x).
\]
\end{lemma}
\begin{proof}
For the spectral decomposition
\[
 x=\int_{-\infty}^\infty\lambda\,e(d\lambda),
\]
put
\[
 p_n=e([-n,n])\uparrow\1,
\]
and let, as before, $x_{[n]}$ be the truncation
\[
 x_{[n]}=\int_{-n}^n\lambda\,e(d\lambda).
\]
Let $\ro\in\M_*$ have the density $h_\om x$. Then
\begin{align*}
 \tau(h_\om x)&=\ro(\1)=\lim_{n\to\infty}\ro(p_n)=\lim_{n\to\infty}\tau(h_\om xp_n)\\
 &=\lim_{n\to\infty}\tau(h_\om x_{[n]})=\lim_{n\to\infty}\om(x_{[n]})=\om(x). \qedhere
\end{align*}
\end{proof}
Furthermore, the following representation of $\om(x)$ holds true.
\begin{lemma}\label{om1}
Let $x$ be a selfadjoint operator affiliated with $\M$. Then
\[
 \om(x)=\langle\Lambda\big(h_\om^{\frac{1}{2}}\big)|\pi(x)\Lambda\big(h_\om^{\frac{1}{2}}\big)\rangle_\Ha
 =\langle\Lambda\big(h_\om^{\frac{1}{2}}\big)|\pi'(x)\Lambda\big(h_\om^{\frac{1}{2}}\big)\rangle_\Ha.
\]
\end{lemma}
\begin{proof}
Assume first that $x\geq0$. Since truncation is a Borel function, we have for the representation $\pi$
\[
 \pi(x_{[n]})=\pi(x)_{[n]},
\]
and the formulae \eqref{f,om} yield
\begin{align*}
 \om(x)&=\lim_{n\to\infty}\om(x_{[n]})=\lim_{n\to\infty}\langle\Lambda\big(h_\om^{\frac{1}{2}}\big)|\pi(x_{[n]})\Lambda\big(h_\om^{\frac{1}{2}}\big)\rangle_\Ha\\
 &=\lim_{n\to\infty}\langle\Lambda\big(h_\om^{\frac{1}{2}}\big)|\pi(x)_{[n]}\Lambda\big(h_\om^{\frac{1}{2}}\big)\rangle_\Ha
 =\langle\Lambda\big(h_\om^{\frac{1}{2}}\big)|\pi(x)\Lambda\big(h_\om^{\frac{1}{2}}\big)\rangle_\Ha.
\end{align*}
For arbitrary selfadjoint $x$, we obtain the result by the decomposition $x=x^+-x^-$ and the formula $\pi(x^\pm)=\pi(x)^\pm$.

The result for the antirepresentation $\pi'$ is obtained in the same way taking into account the formulae \eqref{f,om1}.
\end{proof}
Coming back to the definition of $I(\om,\f)$, observe that an attempt to define it by the formula
\begin{align*}
 I(\om,\f)&=\tau(h_\om\log h_\om-h_\om\log h_\f)\\
 &=\tau(h_\om(\log h_\om-\log h_\f))=\om(\log h_\om-\log h_\f)
\end{align*}
fails because the operator $\log h_\om-\log h_\f$ need not be either densely defined nor affiliated with $\M$. However, we can, in accordance with our previous considerations, define the information by the formula
\begin{equation}\label{inf}
 \begin{aligned}
  I(\om,\f)&=`\tau(h_\om\log h_\om)-\tau(h_\om\log h_\f)'\\
  &=\om(\log h_\om)-\om(\log h_\f)\\
  &=\tau\big(h_\om^{\frac{1}{2}}(\log h_\om)h_\om^{\frac{1}{2}}\big)-\tau\big(h_\om^{\frac{1}{2}}(\log h_\f)h_\om^{\frac{1}{2}}\big),
 \end{aligned}
\end{equation}
under the condition that the right hand side of the above formula is well-defined. As it was remarked earlier, in the finite dimensional case we have
\[
 S(\f,\om)=I(\om,\f).
\]
Our aim is to obtain this equality for an arbitrary semifinite von Neumann algebra.
\begin{theorem}
Let $\f$ and $\om$ be normal states on a semifinite von Neumann algebra $\M$ with a normal semifinite faithful trace $\tau$, and let $\s^\M(\om)\leq\s^\M(\f)$.
Assume that either the Segal entropy $H(\om)$ of $\om$ is finite or that $\om(\log h_\f)$ is finite. Then
\begin{equation}\label{S=I}
 S(\f,\om)=I(\om,\f).
\end{equation}
\end{theorem}
\begin{proof}
Assume first that $H(\om)$ is finite. The relations \eqref{pi-1} and \eqref{pi0} yield the equality
\[
 -\log\pi'\big(h_\om^{\frac{1}{2}}\big)=\log\pi'\Big(\widetilde{h_\om}^{\frac{1}{2}}\Big),
\]
Using Lemma \ref{om1}, property \eqref{pi'1} of the antirepresentation $\pi'$, and the formula (11'), we obtain
\begin{equation}\label{entr}
 \begin{aligned}
  H(\om)&=\tau(h_\om\log h_\om)=\om(\log h_\om)\\
   &=\langle\Lambda\big(h_\om^{\frac{1}{2}}\big)|\pi'\big(\log h_\om\big)\Lambda\big(h_\om^{\frac{1}{2}}\big)\rangle_\Ha\\
   &=\langle\Lambda\big(h_\om^{\frac{1}{2}}\big)|\log\pi'\big(h_\om\big)\Lambda\big(h_\om^{\frac{1}{2}}\big)\rangle_\Ha\\
   &=\langle\Lambda\big(h_\om^{\frac{1}{2}}\big)|2\log\big(\pi'\big(h_\om\big)^{\frac{1}{2}}\big)\Lambda\big(h_\om^{\frac{1}{2}}\big)\rangle_\Ha\\
   &=-2\langle\Lambda\big(h_\om^{\frac{1}{2}}\big)|\big(-\log\pi'\big(h_\om^{\frac{1}{2}}\big)\big)\Lambda\big(h_\om^{\frac{1}{2}}\big)\rangle_\Ha\\
   &=-2\langle\Lambda\big(h_\om^{\frac{1}{2}}\big)|\log\pi'\Big(\widetilde{h_\om}^{\frac{1}{2}}\Big)\Lambda\big(h_\om^{\frac{1}{2}}\big)\rangle_\Ha\\
   &=-2\int_{-\infty}^{\infty}\log v(\lambda)\,\|m(d\lambda)\Lambda\big(h_\om^{\frac{1}{2}}\big)\|_\Ha^2.
 \end{aligned}
\end{equation}
It is known that the relative entropy $S(\f,\om)$ exists and is either finite or equals $+\infty$, so from the formula \eqref{S} it follows that the function $\log(uv)=\log u+\log v$ is either integrable with respect to the measure $\|m(\cdot)\Lambda\big(h_\om^{\frac{1}{2}}\big)\|_\Ha^2$ or its integral equals $-\infty$. Since the function $\log v$ is integrable, which follows from the formula \eqref{entr}, we get that the function $\log u$ is either integrable or its integral equals $-\infty$. Using again Lemma \ref{om1}, and the property \eqref{pi1} of the representation $\pi$, we obtain
\begin{equation}\label{log}
 \begin{aligned}
  \om(\log h_\f)&=2\,\om\big(\log h_\f^{\frac{1}{2}}\big)\\
  &=2\langle\Lambda\big(h_\om^{\frac{1}{2}}\big)|\pi\big(\log h_\f^{\frac{1}{2}}\big)\Lambda\big(h_\om^{\frac{1}{2}}\big)\rangle_{\Ha}\\
  &=2\langle\Lambda\big(h_\om^{\frac{1}{2}}\big)|\log\pi\big(h_\f^{\frac{1}{2}}\big)\Lambda\big(h_\om^{\frac{1}{2}}\big)\rangle_{\Ha}\\
  &=2\int_{-\infty}^{\infty}\log u(\lambda)\,\|m(d\lambda)\Lambda\big(h_\om^{\frac{1}{2}}\big)\|_\Ha^2.
 \end{aligned}
\end{equation}
Consequently, by virtue of the formula \eqref{S}, we get
\begin{equation}\label{=}
 \begin{aligned}
  I(\om,\f)=&\om(\log h_\om)-\om(\log h_\f)\\
  =&-2\int_{-\infty}^{\infty}\log v(\lambda)\,\|m(d\lambda)\Lambda\big(h_\om^{\frac{1}{2}}\big)\|_\Ha^2+\\
  &-2\int_{-\infty}^{\infty}\log u(\lambda)\,\|m(d\lambda)\Lambda\big(h_\om^{\frac{1}{2}}\big)\|_\Ha^2\\
  =&-2\int_{-\infty}^{\infty}\log (u(\lambda)v(\lambda))\,\|m(d\lambda)\Lambda\big(h_\om^{\frac{1}{2}}\big)\|_\Ha^2\\
  &=S(\f,\om),
 \end{aligned}
 \end{equation}
proving the claim.

Now for finite $\om(\log h_\f)$, we proceed in virtually the same way: the equality \eqref{log} yields the integrability of the function $\log u$, so the function $\log v$ is either integrable or its integral equals $-\infty$, and the derivation as in the formula \eqref{=} holds.
\end{proof}
\begin{remark}
In the analysis above, one thing is missing. Namely, we have proved equality between the relative entropy and the information between states under the assumption that either \emph{Segal's entropy of the state $\om$ is finite} or that $\om(\log h_\f)$ \emph{is finite}. However, having only the relative entropy, apparently nothing can be said about the existence of the information. To clarify the situation, it would be interesting either to find an example where the relative entropy exists and the information does not or to prove that the existence of the relative entropy implies the existence of the information.
\end{remark}
To illustrate our considerations, look at the following basic example. Let $\M$ be the full algebra of all bounded linear operators on a separable Hilbert space with the canonical trace $\tr$, and let $\f$ and $\om$ be normal (non-necessarily normalised) states on $\M$. Their densities $h_\f$ and $h_\om$ are positive trace-class operators with spectral decompositions
\[
 h_\f=\sum_{n=1}^\infty\alpha_ne_n, \qquad h_\om=\sum_{n=1}^\infty \beta_nf_n,
\]
where $e_n$ and $f_n$ are finite-dimensional projections such that
\[
 \sum_{n=1}^\infty\alpha_n\tr e_n<\infty, \qquad \sum_{n=1}^\infty\beta_n\tr f_n<\infty.
\]
The space $L^2(\M,\tr)$ consists of the Hilbert-Schmidt operators. We have
\begin{align*}
 \Delta(\f,\om)^{\frac{1}{2}}&=\overline{\pi(h_\f^{\frac{1}{2}})\pi'(h_\om^{-\frac{1}{2}})}\\
 &=\sum_{i,j=1}^\infty\sqrt{\frac{\alpha_i}{\beta_j}}\,\pi(e_i)\pi'(f_j)=\sum_{i,j=1}^\infty\sqrt{\frac{\alpha_i}{\beta_j}}\,m_{ij},
\end{align*}
where
\[
 m_{ij}=\pi(e_i)\pi'(f_j)
\]
are projections. Moreover,
\[
 m_{ij}m_{kr}=\pi(e_i)\pi'(f_j)\pi(e_k)\pi'(f_r)=\pi(e_ie_k)\pi'(f_rf_j)=\delta_{ik}\delta_{jr}m_{ij},
\]
showing that $(m_{ij}:i,j=1,2,\dots)$ is a genuine spectral measure. Thus for the relative modular operator, we obtain the spectral representation
\[
 \Delta(\f,\om)=\sum_{i,j=1}^\infty\frac{\alpha_i}{\beta_j}\,m_{ij}.
\]
Further we have
\begin{align*}
 m_{ij}\Lambda\big(h_\om^{\frac{1}{2}}\big)&=\pi(e_i)\pi'(f_j)\Lambda\big(h_\om^{\frac{1}{2}}\big)=\pi(e_i)\Lambda\big(h_\om^{\frac{1}{2}}f_j\big)\\
 &=\pi(e_i)\Lambda\Big(\sqrt{\beta_j}f_j\Big)=\sqrt{\beta_j}\Lambda(e_if_j),
\end{align*}
which gives
\[
 \|m_{ij}\Lambda\big(h_\om^{\frac{1}{2}}\big)\|_{\Ha}^2=\beta_j\|\Lambda(e_if_j)\|_{\Ha}^2=\beta_j\tr(e_if_j)^*e_if_j=\beta_j\tr e_if_j,
\]
so finally we obtain
\begin{align*}
 S(\f,\om)&=-\langle\Lambda\big(h_\om^{\frac{1}{2}}\big)|\log\Delta(\f,\om)\Lambda\big(h_\om^{\frac{1}{2}}\big)\rangle_{\Ha}\\
 &=-\sum_{i,j=1}^\infty\log\frac{\alpha_i}{\beta_j}\,\|m_{ij}\Lambda\big(h_\om^{\frac{1}{2}}\big)\|_{\Ha}^2\\
 &=-\sum_{i,j=1}^\infty\Big(\log\frac{\alpha_i}{\beta_j}\Big)\,\beta_j\tr e_if_j=\sum_{i,j=1}^\infty\beta_j\log\frac{\beta_j}{\alpha_i}\,\tr e_if_j.
\end{align*}
On the other hand, we have
\[
 h_\om\log h_\om=\sum_{j=1}^\infty\beta_j\log\beta_j\,f_j=\sum_{i,j=1}^\infty\beta_j\log\beta_j\,f_je_i,
\]
and
\[
 h_\om\log h_\f=\sum_{j=1}^\infty\beta_j\,f_j\sum_{i=1}^\infty\log\alpha_i\,e_i=\sum_{i,j=1}^\infty\beta_j\log\alpha_i\,f_je_i,
\]
so assuming for the sake of simplicity that $h_\om\log h_\om$ and $h_\om\log h_\f$ are trace-class (i.e. belong to $L^1(\M,\tr)$), we get
\begin{align*}
 I(\om,\f)&=\tr(h_\om\log h_\om-h_\om\log h_\f)\\
 &=\tr\sum_{i,j=1}^\infty(\beta_j\log\beta_j-\beta_j\log\alpha_i)\,f_je_i
 =\sum_{i,j=1}^\infty\beta_j\log\frac{\beta_j}{\alpha_i}\,\tr f_je_i,
\end{align*}
thus
\[
 S(\f,\om)=I(\om,\f).
\]

\subsection{Quasi-entropies}\label{q-e}
The notion of quasi-entropy was introduced in \cite{OP} and consists in the following. Let $f$ be a continuous function, and let $k\in\M$. A \emph{quasi-entropy} $S_f^k(\f,\om)$ for normal states $\f$, $\om$ on $\M$ is defined as
\[
 S_f^k(\f,\om)=\langle k\xi|f(\Delta(\f,\om))k\xi\rangle,
\]
where $\xi$ is a vector representing the state $\om$. (Note that for $f=-\log$ and $k=\1$, we obtain the Araki relative entropy.) In \cite{OP}, formulae for a number of quasi-entropies are given in finite dimension. It turns out that corresponding formulae can be obtained also in the case of a semifinite von Neumann algebra. First notice that in this case, we have in our fundamental representation
\begin{align*}
 S_f^k(\f,\om)&=\langle\pi(k)\Lambda\big(h_\om^{\frac{1}{2}}\big)|f(\Delta(\f,\om))\pi(k)\Lambda\big(h_\om^{\frac{1}{2}}\big)\rangle_\Ha\\
 &=\langle\Lambda\big(kh_\om^{\frac{1}{2}}\big)|f(\Delta(\f,\om))\Lambda\big(kh_\om^{\frac{1}{2}}\big)\rangle_\Ha.
\end{align*}
We begin with the following counterpart of Lemma \ref{om1}.
\begin{lemma}\label{om11}
Let $x$ be a selfadjoint operator affiliated with $\M$. Then for every $k\in\M$, we have
\[
 (k\om k^*)(x)=\langle\Lambda\big(kh_\om^{\frac{1}{2}}\big)|\pi(x)\Lambda\big(kh_\om^{\frac{1}{2}}\big)\rangle_\Ha,
\]
and
\[
 \ro(x)=\langle\Lambda\big(kh_\om^{\frac{1}{2}}\big)|\pi'(x)\Lambda\big(kh_\om^{\frac{1}{2}}\big)\rangle_\Ha,
\]
where $\ro$ is a normal state on $\M$ with the density $h_\ro=|kh_\om^{\frac{1}{2}}|^2$.
\end{lemma}
\begin{proof}
For $x\in\M$ we have
\begin{align*}
 &\langle\Lambda\big(kh_\om^{\frac{1}{2}}\big)|\pi(x)\Lambda\big(kh_\om^{\frac{1}{2}}\big)\rangle_\Ha
 =\langle\Lambda\big(kh_\om^{\frac{1}{2}}\big)|\Lambda\big(xkh_\om^{\frac{1}{2}}\big)\rangle_\Ha\\
 =&\tau\big(h_\om^{\frac{1}{2}}k^*xkh_\om^{\frac{1}{2}}\big)=\tau\big(h_\om k^*xk\big)=\om(k^*xk)=(k\om k^*)(x),
\end{align*}
and
\begin{align*}
 &\langle\Lambda\big(kh_\om^{\frac{1}{2}}\big)|\pi'(x)\Lambda\big(kh_\om^{\frac{1}{2}}\big)\rangle_\Ha
 =\langle\Lambda\big(kh_\om^{\frac{1}{2}}\big)|\Lambda\big(kh_\om^{\frac{1}{2}}x\big)\rangle_\Ha\\
 =&\tau\big(h_\om^{\frac{1}{2}}k^*kh_\om^{\frac{1}{2}}x\big)=\tau\big(|kh_\om^{\frac{1}{2}}|^2 x\big)=\ro(x).
\end{align*}
The rest of the proof employs truncation and follows the lines of the proof of Lemma \ref{om1}.
\end{proof}
Below we calculate the quasi-entropies as in \cite{OP} in terms of the trace and density operators in a semifinite von Neumann algebra.
\begin{enumerate}
 \item[1.] $f(t)=-\log t$. Then with the state $\ro$ as above
\begin{equation}\label{q-e1}
 \begin{aligned}
  S_f^k(\f,\om)&=\ro(\log h_\om)-(k\om k^*)(\log h_\f)\\
  &=\tau\big(|kh_\om^{\frac{1}{2}}|(\log h_\om)|kh_\om^{\frac{1}{2}}|\big)-\tau\big(h_\om^{\frac{1}{2}}k^*(\log h_\f)kh_\om^{\frac{1}{2}}\big),
 \end{aligned}
\end{equation}
provided that any of the sides in the above equation is well-defined. The derivation of this formula is almost the same as that of the formula \eqref{S=I}. First observe that we have, in the notation as in Section \ref{ss1},
\[
 S_f^k(\f,\om)=-2\int_{-\infty}^\infty\log(u(\lambda)v(\lambda))\,\|m(d\lambda)\Lambda\big(kh_\om^{\frac{1}{2}}\big)\|_{\Ha}^2,
\]
and
\begin{align*}
 (k\om k^*)(\log h_\f)&=\langle\Lambda\big(kh_\om^{\frac{1}{2}}\big)|\pi(\log h_\f)\Lambda\big(kh_\om^{\frac{1}{2}}\big)\rangle_\Ha\\
 &=2\int_{-\infty}^\infty\log u(\lambda)\,\|m(d\lambda)\Lambda\big(kh_\om^{\frac{1}{2}}\big)\|_{\Ha}^2,
\end{align*}
and
\begin{align*}
 \ro(\log h_\om)&=\langle\Lambda\big(kh_\om^{\frac{1}{2}}\big)|\pi'(\log h_\om)\Lambda\big(kh_\om^{\frac{1}{2}}\big)\rangle_\Ha\\
 &=-\langle\Lambda\big(kh_\om^{\frac{1}{2}}\big)|\pi'(\log\widetilde{h_\om})\Lambda\big(kh_\om^{\frac{1}{2}}\big)\rangle_\Ha\\
 &=-2\int_{-\infty}^\infty\log v(\lambda))\,\|m(d\lambda)\Lambda\big(kh_\om^{\frac{1}{2}}\big)\|_{\Ha}^2.
\end{align*}
Assuming, for simplicity, that both terms on the right hand side of the equality \eqref{q-e1} are finite, we get
\begin{align*}
 &\ro(\log h_\om)-(k\om k^*)(\log h_\f)\\
 =&-2\int_{-\infty}^\infty\log v(\lambda))\,\|m(d\lambda)\Lambda\big(kh_\om^{\frac{1}{2}}\big)\|_{\Ha}^2+\\
 \phantom{=}&-2\int_{-\infty}^\infty\log u(\lambda))\,\|m(d\lambda)\Lambda\big(kh_\om^{\frac{1}{2}}\big)\|_{\Ha}^2\\
 =&-2\int_{-\infty}^\infty\log(u(\lambda)v(\lambda))\,\|m(d\lambda)\Lambda\big(kh_\om^{\frac{1}{2}}\big)\|_{\Ha}^2=S_f^k(\f,\om).
\end{align*}
 \item[2.] $f(t)=t^\alpha$, $0<\alpha\leq1$. Then
\begin{equation}\label{q-e2}
 S_f^k(\f,\om)=\tau\big(h_\om^{1-\alpha}k^*h_\f^\alpha k\big).
\end{equation}
For $\alpha\leq\frac{1}{2}$, we have, on account of the formula \eqref{basic1},
\begin{align*}
 S_f^k(\f,\om)&=\langle\Lambda\big(kh_\om^{\frac{1}{2}}\big)|\Delta(\f,\om)^\alpha\Lambda\big(kh_\om^{\frac{1}{2}}\big)\rangle_\Ha\\
 &=\langle\Lambda\big(kh_\om^{\frac{1}{2}}\big)|\Lambda\big(h_\f^\alpha kh_\om^{\frac{1}{2}-\alpha}\big)\rangle_\Ha\\
 &=\tau\big(h_\om^{\frac{1}{2}}k^*h_\f^\alpha kh_\om^{\frac{1}{2}-\alpha}\big)=\tau\big(h_\om^{1-\alpha}k^*h_\f^\alpha k\big).
\end{align*}
For $\alpha>\frac{1}{2}$, we proceed as before
\begin{align*}
 S_f^k(\f,\om)&=\langle\Delta(\f,\om)^{\frac{1}{2}}\Lambda\big(kh_\om^{\frac{1}{2}}\big)|\Delta(\f,\om)^{\alpha-\frac{1}{2}}\Lambda\big(kh_\om^{\frac{1}{2}}\big)\rangle_\Ha\\
 &=\langle\Lambda\big(h_\f^{\frac{1}{2}}k\big)|\Lambda\big(h_\f^{\alpha-\frac{1}{2}}kh_\om^{1-\alpha}\big)\rangle_\Ha\\
 &=\tau\big(k^*h_\f^\alpha kh_\om^{1-\alpha}\big)=\tau\big(h_\om^{1-\alpha}k^*h_\f^\alpha k\big).
\end{align*}
 \item[3.] $f(t)=t\log t$. Let $\ro_1$ and $\ro_2$ be normal states on $\M$ with the densities respectively $h_{\ro_1}=|k^*h_\f^{\frac{1}{2}}|^2$ and $h_{\ro_2}=|h_\f^{\frac{1}{2}}k|^2$. Then
\begin{equation}\label{q-e3}
 \begin{aligned}
  S_f^k(\f,\om)&=\ro_1(\log h_\f)-\ro_2(\log h_\om)\\
  &=\tau\big(|k^*h_\f^{\frac{1}{2}}|(\log h_\f)|k^*h_\f^{\frac{1}{2}}|\big)-\tau\big(|h_\f^{\frac{1}{2}}k|(\log h_\om)|h_\f^{\frac{1}{2}}k|\big),
 \end{aligned}
\end{equation}
provided that any of the sides in the above equation is well-defined.

Observe that for the vector $\Lambda\big(h_\f^{\frac{1}{2}}k\big)$ and $x\in\M$, we have
\begin{align*}
 &\langle\Lambda\big(h_\f^{\frac{1}{2}}k\big)|\pi(x)\Lambda\big(h_\f^{\frac{1}{2}}k\big)\rangle_\Ha
 =\langle\Lambda\big(h_\f^{\frac{1}{2}}k\big)|\Lambda\big(xh_\f^{\frac{1}{2}}k\big)\rangle_\Ha\\
 =&\tau\big(k^*h_\f^{\frac{1}{2}}xh_\f^{\frac{1}{2}}k\big)=\tau\big(h_\f^{\frac{1}{2}}kk^*h_\f^{\frac{1}{2}}x\big)=\tau\big(|k^*h_\f^{\frac{1}{2}}|^2x\big)=\ro_1(x),
\end{align*}
and
\begin{align*}
 &\langle\Lambda\big(h_\f^{\frac{1}{2}}k\big)|\pi'(x)\Lambda\big(h_\f^{\frac{1}{2}}k\big)\rangle_\Ha
 =\langle\Lambda\big(h_\f^{\frac{1}{2}}k\big)|\Lambda\big(h_\f^{\frac{1}{2}}kx\big)\rangle_\Ha\\
 =&\tau\big(k^*h_\f kx\big)=\tau\big(|h_\f^{\frac{1}{2}}k|^2x\big)=\ro_2(x),
\end{align*}
and reasoning as in the proof of Lemma \ref{om1} shows that the formulae above hold for every operator $x$ affiliated with $\M$ such that the right hand sides are well-defined.
Further, we have
\begin{align*}
 S_f^k(\f,\om)&=\langle\Lambda\big(kh_\om^{\frac{1}{2}}\big)|\Delta(\f,\om)\log\Delta(\f,\om)\Lambda\big(kh_\om^{\frac{1}{2}}\big)\rangle_\Ha\\
 &=\langle\Delta(\f,\om)^{\frac{1}{2}}\Lambda\big(kh_\om^{\frac{1}{2}}\big)|(\log\Delta(\f,\om))\Delta(\f,\om)^{\frac{1}{2}}\Lambda\big(kh_\om^{\frac{1}{2}}\big)\rangle_\Ha\\
 &=\langle\Lambda\big(h_\f^{\frac{1}{2}}k\big)|\log\Delta(\f,\om)\Lambda\big(h_\f^{\frac{1}{2}}k\big)\rangle_\Ha,
\end{align*}
and
\begin{align*}
 \ro_1(\log h_\f)&=\langle\Lambda\big(h_\f^{\frac{1}{2}}k\big)|\pi(\log h_\f)\Lambda\big(h_\f^{\frac{1}{2}}k\big)\rangle_\Ha\\
 &=2\int_{-\infty}^\infty\log u(\lambda)\,\|m(d\lambda)\Lambda\big(h_\f^{\frac{1}{2}}k\big)\|_{\Ha}^2,
\end{align*}
\begin{align*}
 \ro_2(\log h_\om)&=\langle\Lambda\big(h_\f^{\frac{1}{2}}k\big)|\pi'(\log h_\om)\Lambda\big(h_\f^{\frac{1}{2}}k\big)\rangle_\Ha\\
 &-\langle\Lambda\big(h_\f^{\frac{1}{2}}k\big)|\pi'(\log\widetilde{h_\om})\Lambda\big(h_\f^{\frac{1}{2}}k\big)\rangle_\Ha\\
 &=-2\int_{-\infty}^\infty\log v(\lambda)\,\|m(d\lambda)\Lambda\big(h_\f^{\frac{1}{2}}k\big)\|_{\Ha}^2.
\end{align*}
Consequently,
\begin{align*}
 \ro_1(\log h_\f)-\ro_2(\log h_\om)&=2\int_{-\infty}^\infty\log(u(\lambda)v(\lambda))\,\|m(d\lambda)\Lambda\big(h_\f^{\frac{1}{2}}k\big)\|_{\Ha}^2\\
 &=\int_{-\infty}^\infty\log(u^2(\lambda)v^2(\lambda))\,\|m(d\lambda)\Lambda\big(h_\f^{\frac{1}{2}}k\big)\|_{\Ha}^2\\
 &=\langle\Lambda\big(h_\f^{\frac{1}{2}}k\big)|\log\Delta(\f,\om)\Lambda\big(h_\f^{\frac{1}{2}}k\big)\rangle_\Ha\\
 &=S_f^k(\f,\om).
\end{align*}
\item[4.] $f(t)=at+b$. Then
\begin{equation}\label{q-e4}
 S_f^k(\f,\om)=a\f(kk^*)+b\om(k^*k).
\end{equation}
Indeed, we have
\begin{align*}
 &S_f^k(\f,\om)=\langle\Lambda\big(kh_\om^{\frac{1}{2}}\big)|(a\Delta(\f,\om)+b)\Lambda\big(kh_\om^{\frac{1}{2}}\big)\rangle_\Ha\\
 =&a\langle\Lambda\big(kh_\om^{\frac{1}{2}}\big)|\Delta(\f,\om)\Lambda\big(kh_\om^{\frac{1}{2}}\big)\rangle_\Ha
 +b\langle\Lambda\big(kh_\om^{\frac{1}{2}}\big)|\Lambda\big(kh_\om^{\frac{1}{2}}\big)\rangle_\Ha\\
 =&a\langle\Delta(\f,\om)^{\frac{1}{2}}\Lambda\big(kh_\om^{\frac{1}{2}}\big)|\Delta(\f,\om)^{\frac{1}{2}}\Lambda\big(kh_\om^{\frac{1}{2}}\big)\rangle_\Ha
 +b\tau\big(h_\om^{\frac{1}{2}}k^*kh_\om^{\frac{1}{2}}\big)\\
 =&a\langle\Lambda\big(h_\f^{\frac{1}{2}}k\big)|\Lambda\big(h_\f^{\frac{1}{2}}k\big)\rangle_\Ha+b\om(k^*k)\\
 =&a\tau(k^*h_\f k)+b\om(k^*k)=a\f(kk^*)+b\om(k^*k).
\end{align*}
\item[5.] \textbf{The skew information.} For $0<p<1$, $\f\in\M_*^+$, and $k\in\M$, the \emph{skew information} $I_p(\f,k)$ in finite dimension is defined as
 \[
  I_p(\f,k)=\f(kk^*)-\tau\big(h_\f^{1-p}k^*h_\f^pk\big).
 \]
Let $f(t)=t-t^p$ for $0<p<1$. Then as in point 2, we obtain
\begin{align*}
 S_f^k(\f,\om)&=\langle\Lambda\big(kh_\om^{\frac{1}{2}}\big)|\Delta(\f,\om)\Lambda\big(kh_\om^{\frac{1}{2}}\big)\rangle_\Ha+\\ &-\langle\Lambda\big(kh_\om^{\frac{1}{2}}\big)|\Delta(\f,\om)^p\Lambda\big(kh_\om^{\frac{1}{2}}\big)\rangle_\Ha\\
 &=\langle\Delta(\f,\om)^{\frac{1}{2}}\Lambda\big(kh_\om^{\frac{1}{2}}\big)|\Delta(\f,\om)^{\frac{1}{2}}\Lambda\big(kh_\om^{\frac{1}{2}}\big)\rangle_\Ha
 -\tau\big(h_\om^{1-p}k^*h_\f^pk\big)\\
 &=\tau(k^*h_\f k)-\tau\big(h_\om^{1-p}k^*h_\f^pk\big)=\f(kk^*)-\tau\big(h_\om^{1-p}k^*h_\f^pk\big),
\end{align*}
and taking $\om=\f$ shows that also in the case of a semifinite von Neumann algebra the skew information may be defined by the formula
\[
 I_p(\f,k)=S_f^k(\f,\f).
\]
\end{enumerate}

\subsection{R\'enyi's relative entropy}
For $\alpha\ne1$, \emph{R\'enyi's relative entropy} $S_\alpha(\om,\f)$ between states $\f$ and $\om$ is defined by the formula
\[
 S_\alpha(\om,\f)=\frac{1}{\alpha-1}\log\langle\xi|\Delta(\f,\om)^{1-\alpha}\xi\rangle,
\]
where $\xi$ is the vector representing the state $\om$,
\[
 \om(x)=\langle\xi|x\xi\rangle, \quad x\in\M.
\]
In the setting of semifinite von Neumann algebras, we have, in our basic representation introduced in Section \ref{rep}, $\xi=\Lambda\big(h_\om^{\frac{1}{2}}\big)$, thus
\begin{equation}\label{R-ent}
 S_\alpha(\om,\f)=\frac{1}{\alpha-1}\log\langle\Lambda\big(h_\om^{\frac{1}{2}}\big)|\Delta(\f,\om)^{1-\alpha}\Lambda\big(h_\om^{\frac{1}{2}}\big)\rangle.
\end{equation}
\begin{theorem}
For every $0\leq\alpha<1$, the following formula holds
\begin{equation}\label{R}
 S_\alpha(\om,\f)=\frac{1}{\alpha-1}\log\tau\big(h_\f^{1-\alpha}h_\om^\alpha\big).
\end{equation}
\end{theorem}
\begin{proof}
First note that the expression \eqref{R-ent} is a little formal because for $\alpha<\frac{1}{2}$, it may happen that $\Lambda\big(h_\om^{\frac{1}{2}}\big)\notin\D(\Delta(\f,\om)^{1-\alpha})$ --- the situation which we already encountered in the definition of Araki's relative entropy. Again it is to be understood as
\begin{align*}
 S_\alpha(\om,\f)&=\frac{1}{\alpha-1}\log\int_0^\infty \lambda^{1-\alpha}\,\langle\Lambda\big(h_\om^{\frac{1}{2}}\big)|e(d\lambda)\Lambda\big(h_\om^{\frac{1}{2}}\big)\rangle_\Ha\\
 &=\frac{1}{\alpha-1}\log\int_0^\infty \lambda^{1-\alpha}\,\|e(d\lambda)\Lambda\big(h_\om^{\frac{1}{2}}\big)\|_\Ha^2
\end{align*}
where
\[
 \Delta(\f,\om)=\int_0^\infty\lambda\,e(d\lambda)
\]
is the spectral decomposition of $\Delta(\f,\om)$. The proof follows from the calculations performed in Section \ref{q-e} point 2, for $k=\1$ and $\alpha$ in place of $1-\alpha$, where the formula
\[
 \langle\Lambda\big(kh_\om^{\frac{1}{2}}\big)|\Delta(\f,\om)^\alpha\Lambda\big(kh_\om^{\frac{1}{2}}\big)\rangle_\Ha=\tau\big(h_\om^{1-\alpha}k^*h_\f^\alpha k\big)
\]
was obtained.
\end{proof}
\begin{remark}
In finite dimension, the formula \eqref{R} has been known for some time. Moreover, it is valid for all $\alpha\ne1$ which is pretty obvious since then there are no restrictions on the domain of the relative modular operator.
\end{remark}

\end{document}